\documentclass[11pt,reqno]{amsart}
\usepackage{amsmath}
\usepackage{amssymb}
\usepackage{amsthm}
\usepackage{amscd, amsfonts, mathrsfs}
\usepackage{amsaddr}
\usepackage{cases}
\usepackage[dvips]{epsfig}
\usepackage{verbatim} 
\usepackage{epsf}
\usepackage[bookmarksnumbered,pdfpagelabels=true,plainpages=false,colorlinks=true,
            linkcolor=black,citecolor=black,urlcolor=black]{hyperref}

\theoremstyle{plain}
\newtheorem{theorem}{Theorem}[section]

\newtheorem{lemma}[theorem]{Lemma}
\newtheorem{proposition}[theorem]{Proposition}

\newtheorem{assumption}[theorem]{Assumption}

\theoremstyle{definition}
\newtheorem{remark}[theorem]{Remark}
\newtheorem{notation}[theorem]{Notation}
\newtheorem{definition}[theorem]{Definition}
\newtheorem{convention}[theorem]{Convention}

\numberwithin{equation}{section}
\allowdisplaybreaks

\newcommand{\cA}{\mathcal A}

\newcommand{\cI}{\mathcal I}

\newcommand{\cT}{\mathcal T}
\newcommand{\cU}{\mathcal U}
\newcommand{\cV}{\mathcal V}
\newcommand{\cW}{\mathcal W}

\newcommand{\cY}{\mathcal Y}

\newcommand{\al}{\alpha}
\newcommand{\be}{\beta}
\newcommand{\ga}{\gamma}
\newcommand{\Ga}{\Gamma}

\newcommand{\ep}{\epsilon}
\newcommand{\la}{\lambda}

\newcommand{\Si}{\Sigma}

\newcommand{\Om}{\Omega}

\newcommand{\RR}{\mathbb R}
\newcommand{\CC}{\mathbb C}

\newcommand{\rar}{\rightarrow}

\newcommand{\tr}{\operatorname{tr}}

\newcommand{\dive}{\operatorname{div}}

\newcommand{\norm}[1]{\left\Vert#1\right\Vert}

\title[Conformal Fluids]{On the existence of solutions and causality for
relativistic viscous conformal fluids}

\author[Disconzi]{Marcelo M. Disconzi}
\address{Department of Mathematics\\
Vanderbilt University\\ Nashville, TN, USA}
\email{marcelo.disconzi@vanderbilt.edu}
\thanks{Marcelo M. Disconzi is partially supported by NSF grant \# DMS-1812826,
by a Sloan Research Fellowship provided by the Alfred P. Sloan foundation,
and by a Discovery grant administered by Vanderbilt University.}

\begin{document}

\maketitle

\begin{abstract}
We consider a stress-energy tensor describing a pure radiation viscous fluid with conformal symmetry
introduced in \cite{BemficaDisconziNoronha}.
We show that the corresponding equations of motions are causal
in Minkowski background and also when coupled to Einstein's equations,
and solve the associated initial-value problem.
\end{abstract}


\tableofcontents

\section{Introduction\label{section_intro}}

Consider the following stress-energy tensor for a relativistic fluid with viscosity:
\begin{align}
\begin{split}
T_{\al \be} & = \frac{4}{3}u_\al u_\be  \ep  + \frac{1}{3} g_{\al\be} \ep
-\eta \pi_\al^\mu \pi_\be^\nu (\nabla_\mu u_\nu + \nabla_\nu u_\mu 
-\frac{2}{3} g_{\mu\nu} \nabla_\la u^\la )
\\
&
+\la (u_\al u^\mu \nabla_\mu u_\be + u_\be u^\mu \nabla_\mu u_\al) 
+ \frac{1}{3} \chi \pi_{\al\be} \nabla_\mu u^\mu + \chi u_\al u_\be \nabla_\mu u^\mu 
\\
& + \frac{\la}{4 \ep }(u_\al \pi^\mu_\be \nabla_\mu \ep + u_\be \pi^\mu_\al \nabla_\mu \ep) 
+ \frac{3 \chi}{4 \ep} u_\al u_\be u^\mu\nabla_\mu \ep + \frac{\chi}{4 \ep}
\pi_{\al\be} u^\mu \nabla_\mu \ep.
\end{split}
\label{conformal_tensor}
\end{align}
Here, $u$ is the four-velocity of fluid particles, normalized so that
\begin{gather}
u^\al u_\al =-1,
\label{normalization_u}
\end{gather}
$\ep$ is the energy density of the fluid, $g$ is a (Lorentzian) metric, $\nabla$ is the Levi-Civita
connection associated with $g$, $\pi_{\al\be} = g_{\al\be} + u_\al u_\be$,
and $\eta$, $\lambda$, and $\chi$ are viscous transport
coefficients --- so that $\eta=\lambda=\chi = 0$ corresponds to an ideal fluid. The 
transport coefficients are non-negative functions of $\epsilon$. Coefficient $\eta$
is the usual coefficient of shear viscosity, whereas $\lambda$ and $\chi$ are related to
relaxation times.  More precisely, while $\lambda$ and $\chi$, differently
than $\eta$, have no analogue in more familiar theories such as classical, non-relativistic
 Navier-Stokes,
their physical meaning can be understood from the derivation of (\ref{conformal_tensor})
from kinetic theory given in \cite{BemficaDisconziNoronha}. In that case, one may interpret $\lambda/(s \theta)$ and $\chi/(s \theta)$, where $s$ is the entropy density and $\theta$ the temperature, as relaxation times that restore causality (since intuitively causality says that the system
needs some time to relax back to equilibrium after a perturbation). See \cite{BemficaDisconziNoronha}
for details.

We are interested in the case of  pure radiation, when the fluid's pressure is given
by $p = \frac{1}{3}\ep$, and, therefore, $p$ has already been eliminated from $T_{\al\be}$.

Above and throughout, we adopt the following:
\begin{convention}
We work in units where $8\pi G = c = 1$, where $G$ is Newton's constant and $c$ is the speed
of light in vacuum. Our signature for the metric is $-+++$. Greek indices run from
$0$ to $3$ and Latin indices from $1$ to $3$.
\label{conventions}
\end{convention}

We shall couple (\ref{conformal_tensor}) to Einstein's equations:
\begin{gather}
R_{\al\be} -\frac{1}{2} R g_{\al\be} + \Lambda g_{\al\be} = T_{\al\be},
\label{EE}
\end{gather}
where $R_{\al\be}$ and $R$ are, respectively, the Ricci and scalar curvature of the metric $g$,
and $\Lambda$ is a constant (the cosmological constant).
We recall that in light of the Bianchi identities, a necessary condition for (\ref{EE}) to hold is that
\begin{gather}
\nabla_\al T^\al_\be = 0.
\label{div_T}
\end{gather}
Naturally, equations (\ref{EE})-(\ref{div_T}) are defined in a four-dimensional
differentiable manifold, the space-time.

We shall establish the following.

\vskip 0.2cm
\textbf{Main Result.} 
\emph{(see Theorems \ref{main_theorem} and \ref{Minkowski_theorem} for precise statements)
Under appropriate conditions on the initial data and the transport coefficients, 
the system of Einstein's equations coupled to (\ref{conformal_tensor})
is causal and admits a unique solution. Causality and uniqueness are here understood
in the usual sense of general relativity. Existence, uniqueness, and causality remain true if
we consider solely
(\ref{div_T}) in Minkowski space-time.
}
\vskip 0.2cm

The tensor (\ref{conformal_tensor}) was introduced\footnote{In \cite{BemficaDisconziNoronha}, 
(\ref{conformal_tensor}) is written in a different form, using the so-called
Weyl derivative (whose definition is given in \cite{BemficaDisconziNoronha};
see \cite{Loganayagam:2008is} for more details) instead
of the covariant derivative. Both expressions agree once the Weyl derivative is expanded in terms
of the covariant derivative.} in \cite{BemficaDisconziNoronha}. As discussed there,  (\ref{conformal_tensor}) is the first example in the literature 
of a stress-energy tensor for relativistic viscous fluids satisfying the following list
of physical requirements: in Minkowski background, equations (\ref{div_T})
are (i) linearly stable with respect to perturbations around homogeneous thermodynamic equilibrium, 
(ii) well-posed, and (iii) causal; (iv) Einstein's equations coupled to 
(\ref{conformal_tensor}) are well-posed and causal;
(v) equations (\ref{div_T}) reduce to the standard Navier-Stokes equations
in the non-relativistic limit;
(vi) an out-of-equilibrium entropy can be defined so that solutions to (\ref{div_T}) satisfy the 
(out of equilibrium) second law
of thermodynamics; and 
{(vii) $T_{\al\be}$ can be derived from microscopic kinetic theory. 

One reason for seeking a stress-energy tensor satisfying the above properties is 
that the traditional forms of the relativistic Navier-Stokes equations fail to be causal
and stable \cite{Hiscock_Lindblom_instability_1985, PichonViscous}, and attempts to construct a relativistic viscous theory satisfying (i)-(vi)
have been limited so far\footnote{It is interesting to note that the seemingly easier task of generalizing the 
non-relativistic Navier-Stokes to Riemannian manifolds is not without problems either, see
\cite{CzubakDisconziChan-NavierManifolds}.}. See 
\cite{DisconziViscousFluidsNonlinearity, Disconzi_Kephart_Scherrer_2015, DisconziKephartScherrerNew,RezzollaZanottiBookRelHydro} for a discussion. 
In \cite{BemficaDisconziNoronha} 
it is also shown that $T_{\al\be}$ yields a well-defined temperature 
in the test-case of the Gubser flow, in contrast to the traditional 
relativistic Navier-Stokes' equations that yield a negative temperature, 
and that a hydrodynamic attractor exists for the dynamics of the Bjorken flow.

Tensor (\ref{conformal_tensor}) describes a conformal fluid.
Loosely speaking, this means that (\ref{conformal_tensor}) is well-behaved under conformal changes
of the metric.
 More precisely, consider a conformal transformation
$g_{\al\be}^\prime = e^{-2\phi} g_{\al\be}$, and the transformed
quantities $u^\prime_\al = e^{-\phi} u_\al$, $\ep^\prime = e^{4\phi} \ep$.
Then the fluid is called conformal if $T_{\al\be}$ is traceless and the corresponding transformed $T^\prime_{\al\be}$ satisfies
\begin{gather}
T^\prime_{\al\be} = e^{2\phi} T_{\al\be}.
\nonumber
\end{gather}
One can show \cite{Baier:2007ix, Bhattacharyya:2008jc} that under these conditions 
\begin{gather}
\nabla^\prime_\al (T^\prime)^\al_\be = e^{4\phi} \nabla_\al T^\al_\be,
\nonumber
\end{gather}
so in particular solutions are preserved by the above transformations.
There exists a large literature on conformal fluids and their applications in physics,
to which the reader is referred for a discussion  (see, e.g., \cite{KodamaHydroApproaches, NoronhaExactAnalytical} and references therein;
for the mathematical background for these references, see \cite{HallWeylManifolds}).
We restrict ourselves to mentioning
that conformal fluids are of importance in the study of the quark-gluon plasma that
forms in high-energy collisions of heavy-ions;  the quark-gluon plasma at very high
temperatures is the prototypical example of a relativistic viscous fluid
with an equation of state of pure radiation.

The definition of conformal fluid, stated above, will play no direct role in this work
per se. 
Rather, we shall use one of its main consequences, namely, that  
for such fluids we have 
\begin{gather}
\chi = a_1 \eta, \lambda = a_2 \eta,
\label{proportional_eta}
\end{gather}
where $a_1$ and $a_2$ are constants. Therefore all transport coefficients are determined once we are given
 $\eta = \eta(\ep)$.

Our main result has previously appeared in \cite{BemficaDisconziNoronha}, 
but the letter format of that manuscript
and the fact that
it was addressed primarily to a physical audience prevented us from presenting several details
of the proof. In particular, the argument in \cite{BemficaDisconziNoronha} may not be entirely satisfactory  for a mathematical 
audience. 

\begin{definition}
For the rest of the the paper, we shall refer to the system of equations (\ref{EE}), with $T_{\al\be}$
given by (\ref{conformal_tensor}) and $u$ satisfying (\ref{normalization_u}), as the
viscous Einstein-conformal fluid (VECF) system. 
\end{definition}

\section{Statement of the results}
We now turn to the precise formulation of the Main Result. We begin by discussing the initial data for 
the VECF system.

\begin{definition}
An initial data set for the VECF system consists of 
a three-dimensional smooth manifold $\Si$,
a Riemannian metric $g_0$  on $\Si$,
a symmetric two-tensor $\kappa$  on $\Si$,
two real-valued functions $\ep_0$ and $\ep_1$ defined on $\Si$, and 
two vector fields $v_0$  and $v_1$ on $\Si$,
such that the Einstein constraint equations are satisfied.
\label{definition_initial_data}
\end{definition}

We recall that the constraint equations are given by the following system
of equations on $\Si$:
\begin{align}
\begin{split}
R_{g_0} - |\kappa|^2_{g_0} - (\tr_{g_0} \kappa)^2 &= 2\rho
\\
\nabla_{g_0} \tr_{g_0} \kappa - \dive_{g_0} \kappa & = j 
\end{split}
\nonumber
\end{align}
where $R_{g_0}$ is the scalar curvature of $g_0$, $\nabla_{g_0}$, $\tr_{g_0}$, $\dive_{g_0}$,
and $| \cdot |_{g_0}$ are the covariant derivative, trace, divergence, and norm with 
respect to $g_0$. The quantities $\rho$ and $j$ are given by 
$\rho = T(n,n)$ and $j = T(n,\cdot)$, where $n$ is the future-pointing unit normal to $\Si$
inside a development of the initial data and $T$ is the stress-energy tensor.

Because $T_{\al\be}$ involves first derivatives of $u$ and $\epsilon$, initial conditions for their time derivatives
have to be given, hence the necessity of two functions and two vector fields. Even though $u$ is a four-vector, it suffices
to specify vector fields on $\Si$, with initial conditions for the non-tangential components of $u$ derived from (\ref{normalization_u})
(see section \ref{section_initial_data}). It is well-known that initial data for Einstein's equations cannot be prescribed
arbitrarily, having to satisfy the associated constraint equations, see, e.g., \cite{HawkingEllisBook}, 
for details.

We can now state our main result. The definition of spaces $G^{s}$ and $G^{m,s}$
is recalled in Appendix \ref{section_appendix_Gevrey}.
We refer the reader to the general relativity literature (e.g.,
\cite{ChoquetBruhatGRBook,HawkingEllisBook, KlainermanNicoloBook, RingstromCauchyBook,WaldBookGR1984})
for the terminology employed in Theorem \ref{main_theorem}. 

\begin{theorem}
Let $\cI = (\Si, g_0, \kappa, \ep_0, \ep_1, v_0, v_1)$ be an initial data set for the VECF system. Assume that $\Si$ is compact with no boundary,
and that $\ep_0 > 0$. Suppose that $\chi$ and $\lambda$ are given by (\ref{proportional_eta}),
where $\eta: (0,\infty) \rar (0,\infty)$ is analytic, and assume that 
 $a_1 = 4$ and $a_2 \geq 4$. Finally, assume that 
the initial data is in $G^{(s)}(\Si)$ for some $1 < s < \frac{17}{16}$.
Then: 

1) There exists a globally hyperbolic development $M$ 
of $\cI$. 

2) $M$ is causal, in the following sense. 
Let $(g,\ep,u)$ be a solution to the VECF system provided by the 
globally hyperbolic development $M$.
For any $p \in M$ in the future of $\Si$, $(g(p), u(p), \epsilon(p))$ depends only 
on $\left. \cI \right|_{i(\Si) \cap J^-(p)}$, where
$J^{-}(p)$ is the causal past of $p$ and
 $i: \Si \rar M$ is the embedding associated with the globally hyperbolic development $M$.

\label{main_theorem}
\end{theorem}

We note that, in the standard PDE language, Theorem \ref{main_theorem} is local in time.
But as usual in general relativity, solutions to Einstein's equations are geometric 
(a solution to Einstein's equations is a Lorentzian manifold) 
and,
in particular, coordinate independent, whereas a statement like ``there exists a $\cT >0$..."
(as in usual local in time results) requires the introduction of coordinates. This is
why the theorem is better stated as the existence of a globally hyperbolic 
development\footnote{We recall that a globally hyperbolic development is, roughly speaking, 
a Lorentzian manifold where Einstein's equations are satisfied and in which $\Si$ embeds
isometrically as a Cauchy surface taking the correct data. We also recall
that once a globally hyperbolic development is shown to exist, one can prove the 
existence of the ``largest" possible global hyperbolic development, i.e., the
maximal globally hyperbolic development of the initial data, which is 
(geometrically) unique. See \cite{KlainermanNicoloBook, RingstromCauchyBook} for details.}.
We assumed that $\Si$ is compact for simplicity, otherwise asymptotic conditions would have to be prescribed. 
The type of asymptotic conditions one would impose had $\Si$ been non-compact depends on the
type of questions one is investigating. For instance, it is customary to require $g_0$ to 
be asymptotically flat, but other conditions, such as asymptotically hyperbolic, are often
used. As for the matter variables, several choices are possible. One can require $v_0$ and $\ep_0$
to approach zero, a constant, or some other specified profile at infinity. The literature
on Einstein's equations with non-compact $\Si$ is vast, and a discussion
of asymptotic conditions can be found, e.g., 
\cite{ChoquetBruhatGRBook, ChruscielManifoldStructureConstraints}
and references therein.
The assumption $\ep_0 > 0$  in Theorem \ref{main_theorem} (which implies
a uniform bound from below away from zero by the compactness of $\Si$)}, 
however, is crucial. This is 
apparent from expression (\ref{conformal_tensor}), but it is
worth mentioning that allowing $\epsilon_0$ to vanish leads to severe technical difficulties even 
in the better studied case of the 
Einstein-Euler system (see
\cite{HadzicShkollerSpeck,JangLeFlochMasmoudi,RendallFluidBodies} for 
the known results and \cite{DisconziRemarksEinsteinEuler} for a discussion;
 in fact,
the difficulties with vanishing density are present already in the non-relativistic case, 
see the discussion in 
\cite{DisconziEbinFreeBoundary3d, Lindblad-FreeBoundaryCompressbile}).
In particular, if we were dealing with a non-compact $\Si$ and had chosen an asymptotic condition
where $\ep_0$ approaches zero, the techniques here employed would not directly apply.
The assumptions $a_1 = 4$
and $a_2 \geq 4$ are technical\footnote{
Other values of $a_1$ and $a_2$ are in fact possible as showed in 
\cite{BemficaDisconziNoronha}, and the proof for these other cases is essentially the same
as showed here. The main difference is how one factors the characteristic determinant.
This different factorization is carried out in \cite{BemficaDisconziNoronha}.
See Remark \ref{remark_factorization}.}, but they are consistent with conditions that guarantee the previously mentioned 
linear stability of (\ref{conformal_tensor}). Note that while our proof is restricted
to the Gevrey class, our result guarantees that 
causality will be automatically satisfied in any function space
where uniqueness can be established. This is relevant in view of the difficulties
of constructing causal theories of relativistic viscous fluids.

Next, we consider the case of a Minkowski background.

\begin{theorem}
Let $T$ be given by 
(\ref{conformal_tensor}) with $g$ being the Minkowski metric.
Suppose that $\chi$ and $\lambda$ satisfy (\ref{proportional_eta}),
with $a_1 = 4$, $a_2 \geq 4$, where $\eta: (0,\infty) \rar (0,\infty)$ is a given analytic function.
Let $\ep_0, \ep_1: \RR^3 \rar \RR$ and $v_0, v_2: \RR^3 \rar \RR^3$ belong to $G^{(s)}(\RR^3)$ for some $1 \leq s < \frac{7}{6}$, and assume
that $\ep_0 \geq C_0 > 0$, where $C_0$ is a constant.

Then, there exists a $\cT>0$, a function $\ep: [0,\cT)\times \RR^3 \rar (0, \infty)$, and a vector field $u: [0,\cT) \times \RR^3 \rar \RR^4$, 
such that $(\ep,u)$ satisfies equations (\ref{normalization_u}) and (\ref{div_T}) in $[0,\cT) \times \RR^3$, $\ep(0,\cdot) = \ep_0$, 
$\partial_0 \ep(0,\cdot) = \ep_1$, $u(0,\cdot) = u_0$, and $\partial_0 u(0,\cdot) = u_1$, where $\partial_0$ is the derivative with
respect to the first coordinate in $[0,\cT)\times \RR^3$. 
This solution belongs to $G^{2,(s)}([0,\cT)\times \RR^3)$ and is unique in this class. 
Finally, the solution is causal, in the following sense. For any $p \in [0,T) \times \RR^3$,
$(\ep(p), u(p))$ depends only on $\left. (\ep_0, \ep_1, v_0, v_1) \right|_{\{ x^0 = 0\}\cap 
J^{-}(p)}$, where $J^-(p)$ is the causal past of $p$ (with respect to the Minkowski metric).
\label{Minkowski_theorem}
\end{theorem}

While formally Theorem \ref{Minkowski_theorem} can not be derived as a corollary
of Theorem \ref{main_theorem}, its validity should come as no surprise once we know the latter 
to be true. In fact, the proof of Theorem \ref{Minkowski_theorem} will be essentially
contained in that of Theorem \ref{main_theorem}, as we shall see. It is nonetheless
useful to state Theorem \ref{Minkowski_theorem} given the importance of 
viscous fluids in Minkowski background for applications.

\begin{remark}
The difference between $s >1$ in Theorem \ref{main_theorem} and $s\geq 1$ in Theorem
\ref{Minkowski_theorem} comes from the fact that in the proof of Theorem \ref{main_theorem}
we work in local coordinates and employ bump functions, which cannot be analytic (case
$s=1$). In Minkowski space, however, we can use global coordinates and analyticity is not
prevented. 
\end{remark}

\section{Proof of Theorem \ref{main_theorem}}

In this section we prove Theorem \ref{main_theorem}, thus we henceforth assume its hypotheses.
We will always denote by $s$ a number in $(1,\frac{17}{16})$, as in the statement of the theorem.
The proof will be split in several parts. Some of the arguments parallel well-known
constructions in general relativity in the smooth setting, but we present them because 
some additional steps are required in the Gevrey class.

\subsection{The equations of motion} Here we write the VECF in coordinates and in a more
explicit form. At this point, we are only interested in writing the equations in a suitable form, thus we assume the validity of (\ref{normalization_u}) and (\ref{EE}) (and consequently (\ref{div_T})),
and derive relations of interest.

As is customary, we shall write (\ref{EE}) in trace-reversed
form and in wave coordinates. More precisely, we consider the reduced Einstein equations given by
\begin{gather}
g^{\mu\nu} \partial^2_{\mu \nu} g_{\al\be} 
= B_{\al\be}(\partial \ep, \partial u, \partial g),
\label{EE_gauge}
\end{gather}
where above and henceforth we adopt the following:

\begin{notation}
We shall employ the letters $B$ and $\widetilde{B}$, 
with indices attached when appropriate, to denote
a general expression depending on at most the number of derivatives indicated
in its argument. For instance, in (\ref{EE_gauge}), $B_{\al\be}$ represents an expression
depending on at most first derivatives of $\ep$, first derivatives of $u$, and first
derivatives of $g$. As another example, $\widetilde{B}(\epsilon, \partial u, \partial^2 g)$
denotes an expression depending on at most zero derivatives of $\epsilon$, 
one derivative of $u$, and two derivatives of $g$. $B$ and $\widetilde{B}$ can vary from expression
to expression. It can be easily verified that $B$ and $\widetilde{B}$ will always be an analytic function
(typically involving only products and quotients) of its arguments.
\label{notation_B}
\end{notation}

Equations (\ref{div_T}) become\footnote{See Appendix \ref{appendix_derivation} for a derivation of 
(\ref{EE_gauge}) and (\ref{div_T_gauge}). }
\begin{align}
\begin{split}
(-\eta  & g^{\al\mu} +(\la-\eta) u^\al u^\mu )\partial^2_{\al \mu}   u^\be 
+ (\la + \chi)u^\be u^\mu \partial^2_{\mu\al} u^\al 
+ \frac{1}{3}(-\eta + \chi) g^{\be \mu} \partial^2_{\mu\al} u^\al 
\\
& + \frac{1}{3}(-\eta + \chi) u^\be u^\mu \partial^2_{\mu \al} u^\al 
+ \frac{1}{4 \ep}u^\be (\la  g^{\al \mu} + (\la + 3\chi  ) u^\al u^\mu )\partial^2_{\al\mu} \ep 
\\
&
+ \frac{1}{4\ep} (\la + \chi) u^\al g^{\be \mu} \partial^2_{\al\mu} \ep 
+ \frac{1}{4\ep} (\la + \chi) u^\be u^\al u^\mu \partial^2_{\al \mu} \ep 
+ \widetilde{B}^\be(\partial u, g) \partial^2 g
\\
&
 = B^\be(\partial \ep, \partial u, \partial g).
\end{split}
\label{div_T_gauge}
\end{align}
The term $\widetilde{B}^\be(\partial u, g) \partial^2 g$, which is linear
in $\partial^2 g$, comes from derivatives of the Christoffel
symbols, after expanding the second covariant derivatives of $u$. This term
is of the form $\widetilde{B}^\be(\partial u, g, \partial^2 g) $ according to Notation
\ref{notation_B}, but we wrote it as $\widetilde{B}^\be(\partial u, g) \partial^2 g$ to emphasize
that we shall consider it as a second order quasi-linear operator on $g$. The particular form of this
operator will not be needed, but it is important that it be included in the principal
part of the system for the derivative counting employed below.

Applying $u^\al u^\mu \nabla_\al \nabla_\mu$ to (\ref{normalization_u}) produces
\begin{gather}
u_\la u^\al u^\mu \partial^2_{\al \mu} u^\la + \widetilde{B}(\partial u, g) \partial^2 g= B(\partial u,\partial g).
\label{normalization_u_gauge}
\end{gather}

We introduce the vector 
\begin{gather}
U = (u^\be, \ep, g_{\al\be}),
\nonumber
\end{gather}
where we adopt the 
obvious notation with $u^\be$ denoting $(u^0,u^1,u^2,u^3)$, etc.; such 
a notation is used throughout, including in the matrices below. We write equations 
(\ref{EE_gauge}), (\ref{div_T_gauge}), and (\ref{normalization_u_gauge}) in matrix form
as
\begin{gather}
\mathfrak{M}(U,\partial) U = \mathfrak{q}(U),
\label{main_system}
\end{gather}
where
\begin{gather}
\mathfrak{M}(U,\partial) =
\left(
\begin{matrix}
m(U,\partial) & b(U,\partial) \\
0 & g^{\mu \nu} \partial^2_{\mu\nu}
\end{matrix}
\right)
\label{matrix_M}
\end{gather}
with
\begin{align}
\begin{split}
m_{00}(U,\partial) = &
(-\eta g^{\al\mu} +(\la -\eta) u^\al u^\mu )\partial^2_{\al\mu}
+  (\la + \chi) u^0 u^\al \partial^2_{0\al } 
\\
&+ \frac{1}{3}(-\eta+\chi)(g^{0\al}
+u^0 u^\al  ) \partial^2_{0 \al},
\end{split}
\nonumber
\end{align}
\begin{align}
\begin{split}
m_{0i}(U,\partial) = &
(\la + \chi) u^0 u^\al \partial^2_{\al i} 
+ \frac{1}{3}(-\eta+\chi) (g^{0\al} + u^0 u^\al) \partial^2_{\al i},
\end{split}
\nonumber
\end{align}
\begin{align}
\begin{split}
m_{i\nu}(U,\partial) = & u^i (\la + \chi) u^\al \partial^2_{\al \nu} 
+ \frac{1}{3}(-\eta + \chi)(g^{i\al} + u^i u^\al )\partial^2_{\al \nu} , \, \nu \neq i,
\end{split}
\nonumber
\end{align}
\begin{align}
\begin{split}
m_{ii}(U,\partial) = &( -\eta g^{\al\mu} + (\la - \eta)u^\al u^\mu )\partial^2_{\al\mu} 
+
u^i (\la + \chi) u^\al \partial^2_{\al i} 
\\
&
+ \frac{1}{3}(-\eta + \chi)(g^{i\al} + u^i u^\al )\partial^2_{\al i},
\\
& \text{ with no sum over } i,
\end{split}
\nonumber
\end{align}
\begin{align}
\begin{split}
m_{\nu4}(U,\partial) = & \frac{1}{4\ep}u^\nu( \la g^{\al\mu} + (\la + 3\chi) u^\al u^\mu )\partial^2_{\al\mu} 
+ \frac{1}{4\ep}(\la +\chi) (u^\al g^{\nu\mu} + u^\nu u^\al u^\mu) \partial^2_{\al \mu}, 
\end{split}
\nonumber
\end{align}
\begin{align}
\begin{split}
m_{4\nu}(U,\partial) & = u_\nu u^\al u^\mu \partial^2_{\al\mu}.
\end{split}
\nonumber
\end{align}
(Recall Convention \ref{conventions}: above we have $1 \leq i \leq 3$.) 
The matrix $b(U,\partial)$ in (\ref{matrix_M}) corresponds to the matrix with the operators 
$\widetilde{B}^\be(\partial u, g) \partial^2 $ and  
$\widetilde{B}(\partial u, g) \partial^2 $  that act on $g$ 
 (see (\ref{div_T_gauge}) and (\ref{normalization_u_gauge})), 
whose explicit form will not be important here.
Finally, $g^{\mu\nu}\partial^2_{\mu\nu}$ in (\ref{matrix_M})
represents the $10 \times 10$ identity
matrix times the operator $g^{\mu\nu}\partial^2_{\mu\nu}$. The vector $\mathfrak{q}(U)$
corresponds to the right-hand side of equations (\ref{EE_gauge}), (\ref{div_T_gauge}), and (\ref{normalization_u_gauge}), i.e.,
\begin{gather}
\mathfrak{q}(U) = (B^\be(\partial \ep, \partial u, \partial g), B(\partial u, g), B_{\al\be}(\partial \ep, \partial u, \partial g) ).
\nonumber
\end{gather}

\subsection{Initial data\label{section_initial_data}}
We now investigate the appropriate initial conditions for (\ref{main_system}). 
We remind the reader that the geometric data in the assumptions of Theorem \ref{main_theorem}
are intrinsic to $\Si$, thus they do not determine full data for the system\footnote{For example,
$g_0$ is a metric on $\Si$ which is a three-manifold; thus, $g_0$ contains only nine (six independent)
components locally, whereas there are sixteen (ten independent) components in the full space-time
metric. Similarly, $\kappa$ does not determine all transversal derivatives of $g$ on $\Si$, and
$v_0$ and $v_1$ determine only the initial three-velocity and its
transversal derivatives, whereas we need the four-velocity $u$ and its tranversal derivatives initially.
These mismatches are, as it is well-known, related to the gauge freedom of Einstein's equations.
See, e.g., \cite{ChoquetBruhatGRBook} for more discussion.}. Hence, we need to complete the given data to a full set of initial data.

Assume
that $\cI$ is given as in the statement of Theorem \ref{main_theorem}. Embed 
$\Si$ into $\RR \times \Sigma$ and consider $p \in \{ 0 \} \times \Si$. We shall
initially obtain a solution in a neighborhood of $p$, hence we prescribe initial data
locally.

Take coordinates $\{ x^\al \}_{\al=0}^3$  in a neighborhood $\cU$ of $p$ 
such that $\{ x^i \}_{i=1}^3$ are coordinates on $\Si$, which we assume to be 
normal coordinates for $g_0$ centered at $p$. We remark that in these coordinates
the initial data will be in $G^{(s)}(\{ x^0 = 0 \}\cap \cU )$. For, by our assumption on $\cI$, there exist local coordinates  $\{ y^i \}_{i=1}^3$ in a 
neighborhood $\cY \subseteq \Si $ of $p$ such that, in these coordinates, the initial data is Gevrey regular. One obtains (short-time)
geodesics starting at $p$ by solving the geodesic equation, which will be an ODE with Gevrey data in the
$\{ y^i \}$  coordinates. Since we can equip Gevrey spaces with a norm, the usual Picard iteration can be applied to solve the
geodesic equation, and hence we obtain solutions that are Gevrey regular and vary within the Gevrey class with the initial data. 
Therefore, the exponential map and, as a consequence, the coordinates $\{ x^i \}$ are Gevrey regular in $\cY$ with respect to the
$\{y^i\}$ coordinates. Expressing the initial data now in $\{ x^i \}$ coordinates, we conclude from standard properties
of composition and products of Gevrey maps (see, e.g., \cite{LionsMagenesVol3}) 
that the initial data is in $G^{(s)}(\{ x^0 = 0 \}\cap \cU)$ in the $\{x^i\}$ coordinates.

We prescribe the following
initial conditions for $g_{\al\be}$ on $\{x^0=0\} \cap \cU$:
\begin{align}
\begin{split}
g_{ij}(0,\cdot) = (g_0)_{ij}, \,
g_{00}(0,\cdot) = -1, \, g_{0i}(0,\cdot) = 0, 
\, 
\partial_0 g_{ij}(0,\cdot) = \kappa_{ij}, 
\nonumber
\end{split}
\end{align}
and $\partial_0 g_{0\al}(0,\cdot)$ is chosen such that $\{x^\al \}$ are wave
coordinates for $g$ at $x^0 = 0$ (which is well-known to always be possible).

For $u^\be$, we prescribe
\begin{align}
\begin{split}
 u^i(0,\cdot) =& \, v_0^i, \,
\, u^0(0,\cdot) = \sqrt{1 + (g_0)_{ij} v_0^i v_0^j}, \,
\partial_0 u^i(0,\cdot) = v_1^i,
\\
\partial_0 u^0(0,\cdot)  = & \, \frac{1}{\sqrt{1 + (g_0)_{ij} v_0^i v_0^j}}
\left( (g_0)_{ij} v_0^j v_1^i 
+ \frac{1}{2} \kappa_{ij} v_0^i v_0^j
+ \frac{1}{2} \partial_0 g_{00}(0,\cdot) (1 + (g_0)_{ij} v_0^i v_0^j) \right.
 \\
 & \left. +\, \partial_0 g_{0i}(0,\cdot) v_0^i \sqrt{1 + (g_0)_{ij} v_0^i v_0^j}  \right).
\end{split}
\nonumber
\end{align}
(Note that the radicands are non-negative because $g_0$ is a Riemannian metric.)
The initial conditions for $u^0$ and $\partial_0 u^0$ have been derived from
(\ref{normalization_u}) and the above initial conditions for $g_{\al\be}$. Finally,
\begin{gather}
\ep(0,\cdot) = \ep_0, \, \partial_0 \ep(0,\cdot) = \ep_1.
\nonumber
\end{gather}

\subsection{Initial conditions for the system in $\RR^4$\label{section_initial_R4}}

Consider the local coordinates introduced in section \ref{section_initial_data}. Via these
coordinates and identifying $p$ with the origin, 
we can regard system (\ref{main_system}) as defined in an open set
$\cU$ of $\RR^4$ containing the origin, with the initial conditions prescribed
on $\{x^0 = 0\} \cap \cU$. Note that we can also take (\ref{main_system}) as
a system of equations on the whole of $\RR^4$, and we therefore do so. We seek 
to extend the initial data to the whole hypersurface $\{ x^0 = 0\}$, thus determining
initial conditions for the system in $\RR^4$.

Let $\cV$ be compactly contained in $\{ x^0 = 0 \} \cap \cU$ and $\cW$ be compactly
contained in $\cV$. Let 
$\varphi: \{ x^0 = 0\} \rar \RR$ be a function in $G^{(s)}(\RR^3)$ such that
$0 \leq \varphi \leq 1$, $\varphi =1$ in $\cW$, and $\varphi = 0$ in the complement 
of $\cV$. Denote by $h$ the Minkowski metric and set, on $\{ x^0 = 0\}$,
\begin{align}
\begin{split}
\mathring{g}_{ij} = \varphi (g_0)_{ij} + (1-\varphi) h_{ij},\
\mathring{g}_{00} = -1, \,
\mathring{g}_{0i} = 0, \,
\partial_0 \mathring{g} = \varphi \kappa_{ij}.
\end{split}
\nonumber
\end{align}
These will be initial conditions for $g_{\al\be}$ (for equations (\ref{main_system}) in $\RR^4$),
with an usual abuse of notation to denote the initial conditions involving
$\partial_0$.
As our coordinates have been chosen
with $\{x^i\}$ normal coordinates for $g_0$ centered at $p$, we have that 
$\mathring{g}_{ij}(0)=h_{ij}$ and the deviations of $\mathring{g}_{ij}$ 
from the Minkowski metric restricted to $\{x^0=0 \} \cap \cU$  are quadratic on the 
coordinates away from the origin. Writing
\begin{gather}
\mathring{g}_{ij} = \varphi (g_0)_{ij} + (1-\varphi) h_{ij}
= h_{ij} + \varphi( (g_0)_{ij} - h_{ij} ),
\nonumber
\end{gather}
we see that, shrinking $\cU$ if necessary and taking into account
our choice for $\mathring{g}_{0\al}$, $\mathring{g}_{\al\be}$ is a perturbation of the Minkowsi metric restricted to $\{ x^0=0 \}$. Therefore, $\mathring{g}_{\al\be}$ defines a Lorentzian
metric.

Next, we introduce 
\begin{gather}
\mathring{u}^i = \varphi v_0^i,\, \partial_0 \mathring{u}^i = \varphi v_1^i,
\nonumber
\end{gather}
with the initial conditions for $\mathring{u}^0$ and $\partial_0 \mathring{u}^0$
obtained by the same formulas as in section (\ref{section_initial_data}), with the appropriate
replacements by $\mathring{u^i}$ and $\mathring{g}$ on the right-hand sides. Finally, set
\begin{gather}
\mathring{\ep} = \varphi \ep_0 + 1-\varphi, \, \partial_0 \mathring{\ep} 
= \varphi \ep_1.
\nonumber
\end{gather}
By the compactness of $\Si$ and the assumption $\ep_0 > 0$, it follows that 
$\ep_0 \geq C$ for some constant $C>0$, thus
\begin{gather}
\mathring{\ep} \geq \min\{ \frac{1}{2} C, \frac{1}{2} \} \geq C^\prime > 0,
\nonumber
\end{gather}
for some constant $C^\prime$.

The initial data for (\ref{main_system}) in $\RR^4$ described in this section will be denoted by 
$\mathring{U}$.

\subsection{Solving the system in $\RR^4$\label{section_solving_R4}}
In this section, we solve system (\ref{main_system}) with the initial conditions
described in section \ref{section_initial_R4} (see Proposition \ref{proposition_existence_R4} below). 
We shall employ the techniques, 
terminology, and notation of Leray-Ohya systems reviewed in the appendix.

\begin{lemma}
Equations (\ref{main_system}) form a Leray system. 
\end{lemma}
\begin{proof}
Write $U$ as
$U = (U^1, U^2)$, with the understanding that $U^1 = (u^\be,\ep) = (u^0,u^1,u^2,u^3,\ep)$
and $U^2 = (g_{\al\be})$.
 Assign to (\ref{main_system})
the following indices:
\begin{align}
\begin{array}{cc}
m_1 = 2, & m_2 = 2, \\
n_1 =0, & n_2 = 0,
\end{array}
\nonumber
\end{align}
where $m_1 = m(U^1) \equiv m(u^\be,\ep)$, 
$m_2 = m(U^2) \equiv m(g_{\al\be})$, 
\begin{align}
\begin{split}
n_1  & =  n(\text{equation (\ref{div_T_gauge})})
\\
& =  n(\text{equation (\ref{normalization_u_gauge})})
\\
& \equiv n(\text{equations corresponding to the first five rows of (\ref{main_system})}),
\end{split}
\nonumber
\end{align}
and
\begin{align}
\begin{split}
n_2 & =  n(\text{equation (\ref{EE_gauge})})
\\
& \equiv n(\text{equations corresponding to the last ten rows of (\ref{main_system})}).
\end{split}
\nonumber
\end{align}
It is understood that we have one index $m_I$ for each unknown of the fifteen unknowns 
and one index $n_J$ for each one of the fifteen equations in (\ref{main_system}). For instance,
by $m_1 = m_1(u^\be,\ep)=2$ we mean $m(u^0)=m(u^1)=m(u^2)=m(u^3)=m(\ep)=2$, and so on.

One readily verifies that with this choice of indices, (\ref{main_system}) has the structure
of a Leray system.
Indeed, we list below
for each row $J$ in (\ref{main_system}) or, equivalently, for each
equation in the system (\ref{EE_gauge}), (\ref{div_T_gauge}), and (\ref{normalization_u_gauge}),
the value of $n_J$;
the highest derivatives 
of each unknown entering in the coefficients and on the right-hand side of the equation; and the  difference 
$m_I - n_J$:
\begin{gather}
\text{rows 1-4} \equiv \text{eq. } (\ref{div_T_gauge}): 
n_1 = 0;
\,
\partial u, \partial \ep, \partial g;
\begin{cases}
m(u) - n_1 \equiv m_1 - n_1 = 2,  \\
m(\ep) - n_1 \equiv m_1 - n_1 = 2,  \\
m(g) - n_1 \equiv m_2 - n_1 = 2,  
\end{cases}
\nonumber
\end{gather}
\begin{gather}
\text{row 5} \equiv \text{eq. } (\ref{normalization_u_gauge}): 
n_1 = 0;
\,
\partial u, \partial g;
\begin{cases}
m(u) - n_1 \equiv m_1 - n_1 = 2,  \\
m(\ep) - n_1 \equiv m_1 - n_1 = 2,  \\
m(g) - n_1 \equiv m_2 - n_1 = 2,  
\end{cases}
\nonumber
\end{gather}
and
\begin{gather}
\text{rows 6-15} \equiv \text{eq. } (\ref{EE_gauge}): 
n_2 = 0;
\,
\partial u, \partial \ep, \partial g;
\begin{cases}
m(u) - n_1 \equiv m_1 - n_2 = 2,  \\
m(\ep) - n_1 \equiv m_1 - n_2 = 2,  \\
m(g) - n_1 \equiv m_2 - n_2 = 2.
\end{cases}
\nonumber
\end{gather}
For example, in equations (\ref{div_T_gauge}), for which $n_1 =0$, we have that the left-hand
side consists of differential operators of order $2$ acting on $(u^\be,\ep)$ ($m(u^\be,\ep)-
n_1 = 2$) and differential operators of order $2$ acting 
on $(g_{\al\be})$ ($m(g_{\al \be})- n_1 = 2$),  whose coefficients depend on at most first derivatives of the unknowns 
($\partial u, \partial \ep, \partial g$, i.e., $m(u^\be,\ep)-
n_1 -1$ and $m(g_{\al\be})- n_1-1$); the right-hand side of (\ref{div_T_gauge}), as the 
coefficients of the differential operators, depends
on at most first derivatives of the unknowns.
\end{proof}

\begin{assumption}
We henceforth make explicit use of  (\ref{proportional_eta}),
with $a_1 = 4$ and $a_2 \geq 4$, in accordance with the assumptions
of Theorem \ref{main_theorem}.
\end{assumption}

For the proof of the next proposition, the reader is reminded of
the Definition \ref{definition_A_s} of $\cA^s(\Si,Y)$, which consists 
of the space of functions sufficiently near the Cauchy data.

\begin{proposition}
There exist a $\cT > 0$, a vector field $u: [0,\cT ) \times\RR^3 \rar  \RR^4$, a function $\ep: [0,\cT)\times \RR^3 \rar (0,\infty)$,
and a Lorentzian metric $g$ defined on $[0,\cT) \times \RR^3$, such that $U = (u^\be, \ep, g_{\al\be})$
satisfies (\ref{main_system}) in $[0,\cT) \times \RR^3$ and takes the initial data $\mathring{U}$
on $\{ x^0 = 0\}$. Moreover, $(u,\ep,g) \in G^{2,(s)}([0,\cT)\times \RR^3)$ and this solution is unique in this class. 
\label{proposition_existence_R4}
\end{proposition}
\begin{proof}
We fix the initial data $\mathring{U}$ as constructed in section \ref{section_initial_R4}
and consider $\widehat{U} = (\widehat{u}^\al, \widehat{\ep}, \widehat{g}_{\al\be}) \in \cA^s(\Si,Y)$.
Shrinking $Y$ if necessary, we can assume that $\widehat{g}_{\al\be}$ is a Lorentzian metric, that
$\widehat{\ep} > 0$, and that $\widehat{u}$ is time-like for $\widehat{g}_{\al\be}$, since
these properties hold for $\mathring{U}$.
Because the coefficients of the matrix of differential operators $\mathfrak{M}(U,\partial)$
depend on at most first derivatives of the unknowns, we can evaluate 
these coefficients on $\widehat{U}$.
Denote the corresponding operator by $\mathfrak{M}(\widehat{U},\partial)$.
The characteristic determinant $P(\widehat{U},\xi)$ of (\ref{main_system}),
evaluated at $\widehat{U}$, is
\begin{align}
\begin{split}
P(\widehat{U}, \xi) = \det \mathfrak{M}(\widehat{U},\xi) = p_1(\widehat{U},\xi) 
p_2(\widehat{U},\xi)  p_3(\widehat{U},\xi) 
 p_4(\widehat{U},\xi) 
\end{split}
\label{char_det}
\end{align}
where\footnote{We remark that compared to \cite{BemficaDisconziNoronha}, polynomial 
$p_3(\widehat{U},\xi)$ looks different. That is because in \cite{BemficaDisconziNoronha}
$\widehat{u}^\la \widehat{u}_\la$ had been replaced by $-1$ in view of (\ref{normalization_u}).
Strictly speaking, we are not allowed to do that since one has to prove that $u$ remains
normalized for positive time, which is done in Lemma \ref{lemma_normalization_u} below, but
this was ignored in \cite{BemficaDisconziNoronha} since there only a sketch of the proof
was presented (see the above Introduction).}
\begin{align}
\begin{split}
p_1(\widehat{U},\xi) \equiv p_1(\xi) = \frac{1}{12 \widehat{\ep}} \eta^4 (\widehat{u}^\mu \xi_\mu)^4,
\end{split}
\label{p1}
\end{align}
\begin{align}
\begin{split}
p_2(\widehat{U},\xi) \equiv p_2(\xi) = & \,
\left[ (a_2-1) ( (\widehat{u}^0)^2 \xi_0^2 + (\widehat{u}^1)^2 \xi_1^2 +(\widehat{u}^2)^2 \xi_2^2
+ (\widehat{u}^3)^2 \xi_3^2) - \xi^\mu \xi_\mu \right. 
\\
& \left. +2(a_2-1)( \widehat{u}^1 \widehat{u}^2 \xi_1 \xi_2 + \widehat{u}^1 \widehat{u}^3 \xi_1 \xi_3 + \widehat{u}^2 \widehat{u}^3 \xi_2 \xi_3 )  \right.
\\
&
\left.+ 2(a_2 -1 )\widehat{u}^0 \xi_0 \widehat{u}^i \xi_i \right]^2,
\end{split}
\label{p2}
\end{align}
\begin{align}
\begin{split}
p_3(\widehat{U},\xi) \equiv p_3(\xi) = & \,
-6( (a_2+5)a_2 + (a_2^2+7a_2 -8)\widehat{u}^\la \widehat{u}_\la  )(\widehat{u}^\mu \xi_\mu)^2 
\\
& 
+6(a_2+2)(1+5 \widehat{u}^\la \widehat{u}_\la ) \xi^\mu \xi_\mu,
\end{split}
\label{p3}
\end{align}
and 
\begin{align}
\begin{split}
p_4(\widehat{U},\xi) \equiv p_4(\xi) = & \,
(\xi^\mu \xi_\mu)^{10},
\end{split}
\label{p4}
\end{align}
and the contractions in these expressions are done with respect to the metric
$\widehat{g}_{\al\be}$.
The computation of $P(\widehat{U},\xi)$, and the corresponding factorization
in the above polynomials, is done through a lengthy and tedious algebraic calculation,
part of which was done with the help of the software 
Mathematica\footnote{See Appendix \ref{section_char_det}.}.
Note that the block diagonal form of $\mathfrak{M}(U,\partial)$ allowed us to compute
the characteristic determinant without providing the specific form of the operators 
$\widetilde{B}^\be(\partial u, g) \partial^2 g$ and $\widetilde{B}(\partial u, g) \partial^2 g$.

It is easy to see that the polynomials $\widehat{u}^\mu \xi_\mu$ and $\xi^\mu \xi_\mu$ are
hyperbolic polynomials as long as $\widehat{g}_{\al\be}$ is a Lorentzian metric
and $\widehat{u}$ is time-like with respect to $\widehat{g}_{\al\be}$. Both
conditions are satisfied in view of the constructions in section \ref{section_initial_R4}.
Therefore, $p_1(\xi)$ is the product of four hyperbolic polynomials (recall that $\widehat{\ep} > 0$ and $\eta(\widehat{\ep})>0$), and $p_4(\xi)$ is
the product of ten hyperbolic polynomials. We now move to analyze $p_2(\xi)$ and
$p_3(\xi)$.

Write $p_2(\xi) = (\widetilde{p}_2(\xi))^2$, where $\widetilde{p}_2(\xi)$ is
the second-degree polynomial between brackets in the definition of $p_2(\xi)$.
We claim that $\widetilde{p}_2(\xi)$ is a hyperbolic polynomial. To show this,
we need to investigate the roots $\xi_0 = \xi_0(\xi_1, \xi_2, \xi_3)$ of
the equation $\widetilde{p}_2(\xi) = 0$. Consider first the case where
$\widetilde{p}_2(\xi)$ is evaluated at the origin, i.e., 
$\widetilde{p}_2(\xi)=\widetilde{p}_2(\widehat{U}(0),\xi)$, and assume for a moment that
$\widehat{g}_{\al\be}(0)$ is the Minkowski metric and that $\widehat{u}^\mu \widehat{u}_\mu = -1$.
In this case, the roots are
\begin{align}
\begin{split}
\xi_{0,\pm} = & \, 
-\frac{1}{1 + (a_2 - 1)(1+\underline{\widehat{u}}^2)}
\left( (a_2 -1 )\underline{\widehat{u}}\cdot \underline{\xi} \sqrt{ 1 + \underline{\widehat{u}}^2} \right.
\\
& 
\left. \pm \sqrt{ (a_2 + (a_2 -1 ) \underline{\widehat{u}}^2 )\underline{\xi}^2  - (a_2 -1 )(\underline{\widehat{u}}\cdot \underline{\xi} )^2 }
\right),
\end{split}
\label{roots_p2}
\end{align}
where $\underline{\widehat{u}}=(\widehat{u}^1,\widehat{u}^2,\widehat{u}^3)$, 
$\underline{\widehat{u}}^2 = (\widehat{u}^1)^2+(\widehat{u}^2)^2+
(\widehat{u}^3)^2$, $\underline{\xi} = (\xi_1,\xi_2, \xi_3)$,
$\underline{\xi}^2 = \xi_1^1 + \xi_2^2 + \xi_3^2$,
 and $\cdot$
is the Euclidean inner product. We see that if  $\underline{\xi} = 0$,  then $\xi_{0,\pm} =0$,  and hence
$\xi=0$. Thus, we can assume $\underline{\xi} \neq 0$.
The Cauchy-Schwarz inequality gives 
$\underline{\widehat{u}}^2  \underline{\xi}^2 - (\underline{\widehat{u}} \cdot \underline{\xi} )^2 \geq 0$, hence
 $\xi_{0,+}$ and $\xi_{0,-}$
are real and distinct for $a_2 \geq 4$. We conclude that $\widetilde{p}_2(\xi)$ is a hyperbolic polynomial at the origin. 
Since the roots of a polynomial vary continuously with the polynomial coefficient, $\widetilde{p}_2(\xi)$ will have two distinct
real roots at any point on $\{x^0 = 0\}$ if $\widehat{g}_{\al\be}$ is sufficiently close to the Minkowski metric and $\widehat{u}^\mu\widehat{u}_\mu$ sufficiently close to $-1$.
We know from section \ref{section_initial_R4} that these last conditions are fulfilled upon taking $\cU$
and $Y$ sufficiently small (recall that $\mathring{g}_{\al\be}(0)$ equals the Minkowski metric.). 
Therefore, $\widetilde{p}_2(\xi)$ is a hyperbolic polynomial, and $p_2(\xi)$ is the product of two hyperbolic polynomials.

We now investigate the roots $\xi_0 = \xi_0(\xi_1, \xi_2, \xi_3)$ of
the equation $p_3(\xi) = 0$. As above, we first consider $p_3(\xi)$
evaluated at the origin and suppose that $\widehat{g}_{\al\be}(0)$ is the Minkowski metric
and that $\widehat{u}^\mu \widehat{u}_\mu = -1$, 
which produces 
\begin{align}
\begin{split}
\xi_{0,\pm} = & \,  
\frac{1}{ -2(2+a_2)-(a_2-4)(1+ \underline{\widehat{u}}^2)}
\left((a_2 - 4) \underline{\widehat{u}} \cdot \underline{\xi }\sqrt{1+\underline{\widehat{u}}^2}
\right.
\\ 
&
\left.\pm \sqrt{2}\sqrt{  ( 3a_2(2+a_2)+(a_2^2-2 a_2-8) \underline{\widehat{u}}^2 ) \underline{\xi}^2
-(a_2^2-2 a_2-8 )(\underline{\widehat{u}} \cdot \underline{\xi} )^2
}
\right).
\end{split}
\nonumber
\end{align}
As above, we can assume $\underline{\xi} \neq 0$, and the Cauchy-Schwarz inequality again gives 
$\underline{\widehat{u}}^2  \underline{\xi}^2 - (\underline{\widehat{u}} \cdot \underline{\xi} )^2 \geq 0$. We
readily verify that $(a_2^2-2 a_2-8 ) \geq 0$ and $3a_2(2+a_2) > 0$ for $a_2 \geq 4$. Therefore,
$\xi_{0,+}$ and $\xi_{0,-}$
are real and distinct, and $p_3(\xi)$ is a hyperbolic polynomial at the origin. As above, this implies that $p_3(\xi)$
is a hyperbolic polynomial.

We conclude that $P(\widehat{U},\xi)$ is the product of four degree one (i.e., $p_1(\xi)$),
two degree two (i.e., $p_2(\xi)$), one degree two (i.e., $p_3(\xi)$), and ten degree two (i.e., $p_4(\xi)$) hyperbolic polynomials. 
The Gevrey index of (\ref{main_system}) is thus $\frac{17}{16}$ (see
Remark \ref{remark_Gevrey_index}). Recall that $1 < s < \frac{17}{16}$ by assumption.

Since
$m_I - n_J = 2$ for all $I,J$, and $\sum_I m_I - \sum_J n_J \geq 2$, we have verified the conditions
of Theorem \ref{Leray_Ohya_diagonal} in the appendix. Hence we obtain the diagonalized system
\begin{gather}
\widetilde{\mathfrak{M}}(U,\partial) U = \widetilde{\mathfrak{q}}(U),
\label{system_diagonal}
\end{gather}
where $\widetilde{\mathfrak{M}}(U,\partial)$ is a diagonal matrix whose entries are differential operators
of order 30 (the order of the characteristic determinant, see the appendix) whose coefficients depend on at most 29 derivatives of $U$, 
and $\widetilde{\mathfrak{q}}(U)$ contains
the all the lower order terms.
We want to invoke Theorem \ref{Leray_Ohya_theorem_existence} to solve (\ref{system_diagonal}). To do so, we need to provide initial conditions for (\ref{system_diagonal}).
Since our goal is to obtain a solution to (\ref{main_system}) out of a solution to (\ref{system_diagonal}), such initial conditions
need to be compatible with solutions to (\ref{main_system}).

We shall show that all derivatives of $U$, restricted to $\{x^0 = 0 \}$, can be formally computed from (\ref{main_system}) and written
in terms of the initial data. In particular, initial conditions to (\ref{system_diagonal}) compatible with (\ref{main_system}) can be
determined. As usual in these situations, 
it suffices to show that we can inductively compute $\partial^k_0 U$ on $\{x^0 = 0\}$ as the tangential derivatives
$\partial_i$ can always be computed.

From (\ref{EE_gauge}), we can determine $\left. \partial^2_0 g_{\al\be} \right|_{\{x^0=0\}}$ in terms
of the initial data $\mathring{U}$. Using the result into (\ref{div_T_gauge}), we can write 
$\widetilde{B}^\be(\partial u, g) \partial^2 g$ restricted to $\{x^0 = 0 \}$ in terms of $\mathring{U}$. Equations
(\ref{div_T_gauge}) and (\ref{normalization_u_gauge}) then give
\begin{gather}
\mathfrak{a} 
\left(
\begin{matrix}
\partial^2_0 u^\be \\
\partial^2_0 \ep
\end{matrix}
\right)
= \mathfrak{b},
\nonumber
\end{gather}
where $\mathfrak{b}$ can be written in terms of the initial data on $\{ x^0 = 0\}$, and the matrix $\mathfrak{a}$ is 
the matrix of the coefficients of the terms $\partial^2_0 u^\be$ and $\partial^2_0 \ep$ in equations (\ref{div_T_gauge}) and (\ref{normalization_u_gauge}). At the origin, where $\mathring{g}_{\al\be}(0)$ equals the Minkowski metric, the determinant of 
$\mathfrak{a}$ is
\begin{gather}
\frac{\eta^4}{\ep_0}( 1 + \underline{\mathring{u}}^2)^2( 3a_2 + (a_2 - 4) \underline{\mathring{u}}^2)
(a_2 + (a_2 -1)  \underline{\mathring{u}}^2)^2,
\nonumber
\end{gather}
which is never zero for $a_2 \geq 4$ (recall that $\ep_0 > 0$ and $\eta(\ep_0)>0$). Invoking once more the fact that
$\mathring{g}_{\al\be}$ is a perturbation of the Minkowski metric, we conclude that $\left. \det(\mathfrak{a}) \right|_{\{ x^0 = 0\} }$ never 
vanishes. We can thus invert $\mathfrak{a}$ and write $\partial^2_0 u^\be$ and $\partial^2_0 \ep$ at $x^0=0$ in terms of 
$\mathring{U}$.

It is clear that we can continue this process: differentiate (\ref{EE_gauge}) with respect to $\partial_0$ to determine
$\left. \partial^3_0 g_{\al\be} \right|_{\{x^0=0\}}$; differentiate (\ref{div_T_gauge}) and (\ref{normalization_u_gauge})
with respect to $\partial_0$, use $\left. \partial^3_0 g_{\al\be} \right|_{\{x^0=0\}}$ to eliminate the resulting terms
$\widetilde{B}^\be(\partial u, g) \partial^3 g$ and $\widetilde{B}(\partial u, g) \partial^3 g$, 
and then solve for  $\partial^3_0 u^\be$ and $\partial^3_0 \ep$ at $x^0=0$
(notice that the matrix $\mathfrak{a}$ remains unchanged). Inductively, we can determine all derivatives $\partial_0^k U$
on $\{ x^0 = 0 \}$, $k=2,3, \dots$, in terms of $\mathring{U}$. Moreover, $\left. \partial_0^k U\right|_{\{x^0 = 0\}}$ 
are analytic expressions of $\mathring{U}$ and,
therefore, the initial conditions for (\ref{system_diagonal}) determined in this fashion will be in $G^{(s)}$.

The initial data for (\ref{system_diagonal}), denoted $\mathring{\widetilde{U}}$, consists of the original 
initial data $\mathring{U}$ for (\ref{main_system}), and the values of $\left. \partial^k_0 U \right|_{\{x^0=0\}}$ determined
by the above procedure for $k=2,\dots,29$.

\begin{remark}
The above procedure determines all derivatives of $U$, evaluated at $x^0 = 0$, in terms of the initial conditions
$\mathring{U}$. It follows that if the initial data $\mathring{U}$ is analytic, a well-known argument using power
series can be employed to construct an analytic solution to (\ref{main_system}) in a neighborhood of $\{x^0 = 0\}$.
These techniques for construction of analytic solutions, however, say nothing about causality.
\label{remark_solution_analytic}
\end{remark}

Having supplied (\ref{system_diagonal}) with  appropriate initial conditions, we can now invoke Theorem \ref{Leray_Ohya_theorem_existence} to conclude 
the following. There exist a $\widetilde{\cT} > 0$, a vector field $u: [0,\widetilde{\cT} ) \times \RR^3 \rar \RR^4$, a function $\ep: [0,\widetilde{\cT})\times\RR^3 \rar  (0,\infty)$,
and a Lorentzian metric $g$ defined on $[0,\widetilde{\cT}) \times \RR^3$, such that $U = (u^\be, \ep, g_{\al\be})$
satisfies (\ref{system_diagonal}) in $[0,\widetilde{\cT}) \times \RR^3$ and takes the initial data $\mathring{\widetilde{U}}$
on $\{ x^0 = 0\}$. Moreover, $(u,\ep,g) \in G^{2,(s)}([0,\widetilde{\cT})\times \RR^3)$ and this solution is unique in this class. 

(We note that in invoking Theorem \ref{Leray_Ohya_theorem_existence}, we are using that the intersections of the cones
determined by the polynomials $p_i(\xi)$ have non-empty interiors (recall definition \ref{definition_Leray_Ohya_hyperbolic}).
This follows from the above expressions, but it can also be verified from the explicit computations in section
\ref{section_causality}.)

The conclusions that $\ep > 0$ and $g$ is a Lorentzian metric follow by continuity in the $x^0$ variable, since
these conditions are true at $x^0 = 0$. 

Now we move to obtain a solution to (\ref{main_system}) in $\RR^4$. The argument is similar to the one in
\cite{Lichnerowicz_MHD_book}, thus we shall go over it briefly.

Let $\{ \mathring{U}_k \}_{k=1}^\infty$ be a sequence of analytic initial conditions for the system (\ref{main_system})
converging in $G^{(s)}(\{ x^0 = 0 \} )$ to $\mathring{U}$. For each $k$, let $(u_k, \ep_k, g_k)$ be the analytic solution 
to (\ref{main_system}), defined in a neighborhood of $\{x^0 = 0 \}$, and taking on the initial data $\mathring{U}_k$
(see Remark \ref{remark_solution_analytic}).
Let $\mathring{\widetilde{U}}_k$ be the initial data for (\ref{system_diagonal}) obtained from $\mathring{U}_k$ and 
compatible with (\ref{main_system}), i.e., 
the one derived by the inductive procedure previously described. Then, $\mathring{\widetilde{U}}_k \rar \mathring{\widetilde{U}}$
in $G^{(s)}(\{ x^0 = 0 \} )$.
In light of the compatibility of $\mathring{\widetilde{U}}_k$,
and because (\ref{system_diagonal}) was derived from (\ref{main_system}) via diagonalization, 
the solutions $(u_k, \ep_k, g_k)$ also satisfy (\ref{system_diagonal}). Furthermore, this solution to (\ref{system_diagonal})
also agrees with the one given by Theorem \ref{Leray_Ohya_theorem_existence} 
(since this theorem also applies for analytic data, i.e., $s=1$). The 
energy-type of estimates proved by Leray and Ohya \cite{LerayOhyaNonlinear}
guarantee then that 
 $(u_k, \ep_k, g_k) \rar (u,\ep, g)$ in $G^{(s)}$ and that $(u,\ep,g)$ satisfy the original system (\ref{main_system}). By construction,
 $(u,\ep,g)$ take on the initial data $\mathring{U}$.
 \end{proof}

\begin{remark}
The initial conditions for the VECF system have to satisfy the Einstein constraint equations 
(recall Definition \ref{definition_initial_data}). The initial conditions $\mathring{U}$ satisfy the constraints
in the region $\cW$ in light of the way that $\mathring{U}$ was constructed out of $\left. \cI \right|_{\cU}$. This is, naturally,
necessary for the eventual construction of a full solution to the VECF system. However, purely from the point of view
of (\ref{main_system}) in $\RR^4$, initial condition can be prescribed freely, i.e., they do not have to satisfy any constraints. 
Therefore, the existence of the analytic initial data $\mathring{U}_k$ follows simply by the density
of analytic functions in $G^{(s)}$. Also by density, we can guarantee that the components $(\mathring{\ep}_0)_k$ and 
$(\mathring{g}_{\al\be})_k$ in $\mathring{U}_k$ satisfy $(\mathring{\ep}_0)_k > 0$ and that 
$(\mathring{g}_{\al\be})_k$ is a Lorentzian metric.
\end{remark}

\begin{remark}
The above calculations involving $(a_2^2-2 a_2-8 ) \geq 0$ show why we have the technical assumption $a_2 \geq 4$.
As our calculations were presented already with $a_1 = 4$ in place, they do not reveal the reason for this assumption, 
which as follows.
Computing the characteristic determinant with general $a_1$ produces a very complicated expression with some terms
proportional to $a_1 - 4$. These terms vanish when $a_1 = 4$, and the corresponding expression simplifies to (\ref{char_det}).  This can be seen explicitly in Appendix \ref{section_char_det}.
\end{remark}

\subsection{Causality\label{section_causality}} Having obtained solutions, we now investigate the causality of equations (\ref{main_system}).
As in section \ref{section_solving_R4}, we use results and terminology recalled in the appendix.

\begin{lemma}
The solution $U = (u,\ep,g)$ to (\ref{main_system}) given in Proposition \ref{proposition_existence_R4} is causal, in the following
sense.
For any $x \in [0,\cT) \times \RR^3$, $( u(x), \epsilon(x), g(x))$ depends only 
on $\mathring{U}{\left.\right|_{\{ x^0 = 0\} \cap J^-(x)}}$, where
$J^{-}(x)$ is the causal past of $x$ (with respect to the metric $g$).
\label{lemma_causality}
\end{lemma}
\begin{proof}
Fix  $x \in [0,\cT) \times \RR^3$. The characteristic determinant of (\ref{main_system}) at $x$ is given by
(\ref{char_det}), with the obvious replacement of $\widehat{U}$ by $U$ and evaluated at $x$;
the polynomials $p_i(U(x), \xi) \equiv p_i(x,\xi)$, $i=1,\dots, 4$, are given by expressions (\ref{p1}) to (\ref{p4}),
again with the obvious replacement by $U(x)$. By the same argument used in section
\ref{section_solving_R4} to prove that the $p_i(\xi)$'s are hyperbolic polynomials on $\{ x^0 = 0 \}$, namely, that $g_{\al\be}$
is near the Minkowski metric, we know that the polynomials $p_i(x,\xi)$ are hyperbolic (perhaps after shrinking $\cT$ if necessary).

Denote by $V_i(x)$ the characteristic cone $\{ p_i(x,\xi) = 0 \}$, and by $\Ga_i^{*,\pm}(x)$ the corresponding (forward
and backward) convex cones (on the cotangent space). Let $K^{*,\pm}(x)$ be the (forward and backward) time-like interiors
of the light-cone $\{ g^{\mu \nu}(x) \xi_\mu \xi_\nu = 0 \}$. We need to show that $K^{*,\pm}(x) \subseteq \Ga_i^{*, \pm}(x)$ (see Remark \ref{remark_causality_A_g}).
This is straightforward for $i=1$ and $i=4$.

Assume for a moment that $g$ is the Minkowski metric at $x$ and that $u^\la u_\la = -1$
(note that we have not proved yet that $u$ remains normalized for $x^0>0$). The 
roots of $\{ p_2(x,\xi) = 0 \}$ are given by (\ref{roots_p2}), changing $\widehat{u}$ by 
$u$, which we can write as
\begin{gather}
\xi_{0,\pm} = s_\pm(u,\theta) \sqrt{ \underline{\xi}^2 },
\label{cone_p2_causality}
\end{gather}
where 
\begin{align}
\begin{split}
s_\pm(u,\theta) = & \, 
-\frac{1}{1 + (a_2 - 1)(1+\underline{u}^2)}\left( (a_2 -1 )\sqrt{\underline{u}^2}\cos\theta \sqrt{ 1 + \underline{u}^2}\right.
\\
&
\left. \pm \sqrt{ a_2 + (a_2 -1 ) \underline{u}^2  - (a_2 -1 )\underline{u}^2 \cos^2\theta  }
\right),
\end{split}
\nonumber
\end{align}
$\theta$ is the angle between $\underline{u}$ and $\underline{\xi}$ in $\RR^3$, we used $\underline{u}\cdot \underline{\xi} = \sqrt{\underline{u}^2} \sqrt{\underline{\xi}^2} \cos\theta$, and we omitted the dependence of $u$ and $\theta$ on $x$ for simplicity.

Equation (\ref{cone_p2_causality}) determines the two halves of the characteristic cone $V_2(x)$ in the cotangent space
at $x$. We will have that $K^{*,\pm}(x) \subseteq \Ga_2^{*,\pm}(x)$ if the slopes $s_\pm$ satisfy 
$-1 < s_\pm(u,\theta) < 1$ for each $u$ and $\theta$. To see that this is the case, compute
\begin{gather}
s_\pm(u,0)= s_\pm(u,2\pi) = -\frac{\pm \sqrt{a_2} +(a_2-1)\sqrt{ \underline{u}^2(1+\underline{u}^2)}}{1+(a_2-1)(1+\underline{u}^2)},
\nonumber
\end{gather}
and observe that this expression is always between $-1$ and $1$ for $a_2 \geq 4$. We seek the maxima and minima of 
$s_\pm(u,\theta)$ for $0 < \theta < 2\pi$. Computing the derivative with respect to $\theta$ and solving for $\sin\theta$, we find
$\sin \theta = 0$, i.e., $\theta = \pi$. We readily verify that $-1 < s_\pm(u,\pi) < 1$,
thus 
 $-1 < s_\pm(u,\theta) < 1$.
Since this last condition is open, the result remains true when $g$ is sufficiently close to the Minkowski metric and $u$ sufficiently close to unitary, which 
is the case if $\cT$ is taken sufficiently small. The same argument shows that  $K^{*,\pm}(x) \subseteq \Ga_3^{*,\pm}(x)$, where
again one uses the condition $a_2 \geq 4$.

We conclude that for any $x \in [0,\cT) \times \RR^3$, we have $K^{*,\pm}(x) \subseteq \bigcap_{i=1}^4 \Ga_i^{*,\pm}(x)$,
and the result now follows from Theorem \ref{Leray_Ohya_causality} and Remark \ref{remark_causality_A_g}.
\end{proof}

\begin{remark}
The characteristics associated with $p_1(\xi)$ and $p_4(\xi)$ are of course those of the flow lines and 
gravitational waves. The characteristics associated with $p_3(\xi)$ and $p_2(\xi)$ are interpreted, respectively, 
as sound waves and shear waves. The latter is  sometimes called a second sound wave and is present
also in the M\"uller-Israel-Stewart theory \cite{Hiscock_Lindblom_stability_1983}.
It is useful to compare these characteristics to those of the ideal fluid. In the latter case we
have the flow lines and the sound cone (i.e., the characteristics of the sound waves; 
see \cite{DisconziSpeckRelEulerNull} for a detailed discussion of the role 
of the sound cone in the relativistic Euler equations). Here
it is as if the the sound cone had ``split" into two sound-type characteristics. This 
resembles what happens in magnetohydrodynamics: there two different characteristics are present
for the magnetoacoustic waves, namely, the so-called fast and slow magnetoacoustic waves
(see \cite{AnileBook} for details).
\end{remark}

\subsection{Existence and causality for the system in $\RR \times \Si$\label{section_existence_causality_full}}
Here we show how the solution found in section \ref{section_solving_R4} can be used to construct
a causal solution in a region of $\RR \times \Si$, thus effectively proving Theorem \ref{main_theorem}. 
Recall that we embedded $\Si$ into $\RR \times \Si$.

\begin{remark}
Consider the solution $U=(u,\ep,g)$ to (\ref{main_system}) obtained in Proposition
\ref{proposition_existence_R4}.
 Let $p$ be a point
on $\{ x^0 = 0 \} \times \Si$ and $\cW$ be
as in section \ref{section_initial_R4}.
Let $D_g^+(\cW) \subseteq [0,\cT) \times \RR^3$ be the future domain of dependence of $\cW$
in the metric $g$, where replacing $\cW$ with a smaller set if necessary, we can assume that 
$x^0 < \cT$ for every $(x^0,x^1,x^2, x^3) \in D_g^+(\cW)$.  
In the coordinates on $D^+_g(\cW)$ induced from the coordinates on $[0,\cT)\times \cW$, the solution $U$ is in 
$G^{(2,s)}$ The solution will remain in $G^{(2,s)}$ upon coordinate changes that are 
Gevrey regular \cite{LionsMagenesVol3}. Note that there are plenty of such coordinate changes in that
a smooth manifold always admits a maximal compatible analytic atlas.
\label{remark_Gevrey_regularity}
\end{remark}

\begin{lemma}
It holds that $u^\la u_\la = -1$ in  $D_g^+(\cW) $.
\label{lemma_normalization_u}
\end{lemma}
\begin{proof}
The vector field
$u$ satisfies (\ref{normalization_u_gauge}), whose explicit form is
\begin{gather}
u_\la u^\al u^\mu \nabla_\mu \nabla_\mu u^\la + u^\al \nabla_\al u_\la u^\mu \nabla_\mu u^\la = 0.
\nonumber
\end{gather}
This can be written as
\begin{gather}
\frac{1}{2} u^\al u^\mu \nabla_\al \nabla_\mu (u_\la u^\la ) = 0.
\nonumber
\end{gather}
This is an equation for the scalar $u_\la u^\la$. The operator $u^\al u^\mu \nabla_\al \nabla_\mu$ satisfies the assumptions
of Theorem \ref{Leray_Ohya_theorem_existence}. Therefore, $u_\al u^\al = -1$ in  $D^+_g(\cW)$ if this condition is satisfied initially, which is the case by
construction.
\end{proof}

\begin{lemma}
For every $q \in \Si$ there exists a 
neighborhood $Z_q \subseteq \Si$ of $q$ in $\Si$ and a globally hyperbolic development $M_q$ of 
$\left. \cI \right|_{Z_q}$, where $M_q \subseteq [0,\cT_q) \times \Si$ for some $\cT_q > 0$.
\label{lemma_existence_globally_hyper}
\end{lemma}
\begin{proof}
 Let $p$ be a point
on $\{ x^0 = 0 \} \times \Si$ and $\cW$ be
as in section \ref{section_initial_R4}.
Since the initial conditions $\mathring{U}$
(where $\mathring{U}$ is as in section \ref{section_initial_R4})
 agree on $\cW$ with those from the initial data $\cI$, in view of Lemma \ref{lemma_causality}, we conclude that $U$ is a solution 
 to the reduced Einstein equations
within $D_g^+(\cW)$. It is well-known that a solution to the reduced equations within $D_g^+(\cW)$
is also a solution to the full Einstein's equations if and only if the constraints are satisfied, which is the case by the definition of 
$\cI$. Because $p$ was an arbitrary point, the result is proven.
\end{proof}

We now glue the different $M_q$'s in order to obtain a global (in space) solution. 

\begin{proposition}
Let $q,r \in \Si$,  $Z_q$ and $Z_r$ be neighborhoods of $q$ and $r$ as in lemma \ref{lemma_existence_globally_hyper},
with  globally hyperbolic developments $M_q$ and $M_r$
of $\left. \cI \right|_{Z_q}$ and $\left. \cI \right|_{Z_r}$, respectively, and corresponding solutions $U_q = (u_q, \ep_q, g_q)$ 
and $U_r = (u_r, \ep_r, g_r)$ of the VECF equations. Assume that 
$Z_q \cap Z_r \neq \varnothing$. Then, for any $w \in Z_q \cap Z_r$, there exist neighborhoods $\cU_q$ and $\cU_r$ of $w$ in
$M_q$ and $M_r$, respectively, and a diffeomorphism $\psi: \cU_q \rar \cU_r$ such that $U_q = \psi^*(U_r)$.
\label{proposition_overlapping}
\end{proposition}
\begin{proof}

We shall construct harmonic
coordinates for $g_q$ in a neighborhood of $w$ in $M_q$ as follows. Identifying (a portion of) 
$\Si$ with its embedding in $M_q$, take normal coordinates $(V,\{ y^i \})$ for $g_0$ on $\Si$ centered at $w$, where
$g_0$ comes from the initial data $\cI$. Note that the initial data is Gevrey regular in the $\{ y^i \}$ coordinates 
(see the argument in section \ref{section_initial_data}). We can thus assume that $U_q$ is in $G^{(2,s)}$ (see Remark
\ref{remark_Gevrey_regularity}) 

On $[0,\cT_q )\times V$, 
where $\cT_q > 0$ is some small number such that $U_q$ is defined on  $[0,\cT_q)\times V$,
we introduce coordinates $\{y^\al\}$, $y^0 \in [0,\infty)$. Consider family of initial-value
problems parametrized by $\al$:
\begin{align}
\begin{split}
\nabla^\mu \nabla_\mu  f^{(i)} & = 0, \\
f^{(i)}(0,y^1,y^2,y^3) & = \, y^i, 
\\
\partial_0 f^{(i)}(0,y^1,y^2,y^3) & = \, 0,
\end{split}
\nonumber
\end{align}
and
\begin{align}
\begin{split}
\nabla^\mu \nabla_\mu  f^{(0)} & = 0, \\
f^{(0)}(0,y^1,y^2,y^3) & = \, 0, 
\\
\partial_0 f^{(0)}(0,y^1,y^2,y^3) & = \, 1,
\end{split}
\nonumber
\end{align}
where $\nabla$ is the covariant derivative in the metric $g_q$. 
This problem has a Gevrey regular solution in a neighborhood of $w$ in $[0,\cT_q)\times V$, and 
a standard implicit function type of argument shows that
the functions $x^\al \equiv f^{(\al)}$ define (harmonic) coordinates near $w$. We now consider the change of coordinates 
$x = x(y): [0,\cT_q^{\prime}) \times V^\prime \rar W \subseteq [0,\infty) \times \RR^3$, $x = (x^0, x^1, x^2, x^3)$, where 
$V^\prime$ is a neighborhood of $w$ in $V$, $\cT^{\prime} > 0$ is determined by the foregoing conditions guaranteeing
the existence of the coordinates $\{x^\al\}$, and $W$ is an open set containing the origin. 
Pulling $U_q$ back to $W$ via $x^{-1}$, it follows from these constructions 
that $(x^{-1})^* (U_q)$ satisfies the reduced Einstein equations in $W$. Since $U_q$ originally satisfied  (\ref{normalization_u})
and (\ref{div_T}) as well, we conclude that is a solution to (\ref{main_system}) in $W$.

We can repeat the above argument to obtain wave coordinates $\{z^\al\}$ for $g_r$. Because 
$(V,\{ y^i \})$ is intrinsically determined by $g_0$, and $M_q$ and $M_r$  induce on $Z_q \cap Z_r$ the same initial data,
the map $z$ agrees with $x$ on $\{0\} \times V^\prime$ (in the region where both are defined). From these facts,
we conclude that $(x^{-1})^* (U_q)$ and $(z^{-1})^* (U_r)$ (i) are solutions to (\ref{main_system})
in some domain $[0,t) \times Y \subseteq [0,\infty) \times \RR^3 $ containing the origin, and  (ii) take the same initial data on
$\{0\} \times Y$. 

We have shown that (\ref{main_system}) enjoys uniqueness and causality. Thus, considering possibly a smaller 
region that is globally hyperbolic for both $(x^{-1})^* (g_q)$ and $(z^{-1})^* (g_r)$,
we conclude that
 $(x^{-1})^* (U_q)=(z^{-1})^* (U_r)$, so that $U_q = (z^{-1} \circ x)^*(U_r)$, as desired.
\end{proof}

Using Proposition \ref{proposition_overlapping}, we can now identify overlapping globally hyperbolic developments, 
thus obtaining a globally hyperbolic development of $\cI$ as stated in Theorem \ref{main_theorem}. Causality follows essentially
from Lemma \ref{lemma_causality}: by the foregoing, we can assume that $M$ is diffeomorphic to $[0,\cT)\times \Si$ for some
$\cT>0$. Shrinking $\cT$ if necessary, we reduce the problem to local coordinates, in which case we can
employ wave coordinates. Causality, as stated in Theorem \ref{main_theorem}, 
is preserved by diffeomorphisms, thus the result follows from the causality of the reduced system guaranteed by Lemma 
\ref{lemma_causality}. This finishes the proof of Theorem \ref{main_theorem}.

\section{Proof of Theorem \ref{Minkowski_theorem}}

The proof Theorem \ref{Minkowski_theorem} is essentially contained in the above. In the case of a Minkowski
background, the system reduces to 
\begin{gather}
m(U,\partial) U = \mathfrak{q}(U),
\nonumber
\end{gather}
where $m$ is as in (\ref{matrix_M}), $U = (u^\be,\ep)$ and $\mathfrak{q}(U)$ is as in (\ref{main_system}) with the
appropriate changes for this $5\times 5$ system. The system can be analyzed as
in section \ref{section_solving_R4}. We can do this directly in $\RR^4$, without the complications
of constructing the initial data $\mathring{U}$. The characteristic determinant is given
by $p_1(\xi) p_2(\xi) p_3(\xi)$, where these polynomials are as before, with the simplification that now
we need not carry out any near-Minkowski arguments. Without the matrix $g^{\mu\nu}\partial^2_{\mu\nu}$ coming
from Einstein's equations, the Gevrey index of the system is $\frac{7}{6}$, and analogues of Proposition 
\ref{proposition_existence_R4} and Lemma \ref{lemma_causality} establish the result.

\section*{Acknowledgments} The author would like to thank the anonymous referees
for reading the manuscript carefully and suggesting several improvements.

\appendix

\section{Tools of weakly hyperbolic systems\label{section_appendix_Leray_Ohya}}

For the reader's convenience, we state in this appendix the results about Leray-Ohya systems
(sometimes called weakly hyperbolic systems)
that are used in the proof of Theorem \ref{main_theorem}. These results have been established by
Leray and Ohya in \cite{LerayOhyaLinear,LerayOhyaNonlinear} for the case of systems with diagonal 
principal part, and extended by Choquet-Bruhat in \cite{CB_diagonal} to more general systems.
These
works build upon the classical work of Leray on hyperbolic differential equations \cite{Leray_book_hyperbolic}. 
The reader can consult these references for the  proofs of the results stated below. Further discussion
can be found (without proofs) in \cite{ChoquetBruhatGRBook, DisconziCzubakNonzero,DisconziViscousFluidsNonlinearity}. 
Related results can also be found in \cite{MizohataCauchyProblem}.

We start by recalling some standard notions and fixing the notation that will
be used throughout. 
Given $T>0$, let 
$X = [0,T] \times  \RR^n$.
By  $\partial^k$ we shall denote any $k^{\text{th}}$ order derivative. We shall denote coordinates on $X$ by
$\{ x^\al \}_{\al = 0}^n$, thinking of $x^0 \equiv t$ as the time-variable.
We use the multi-index notation to write
\begin{gather}
\partial^\al 
\equiv
\frac{\partial^{|\al |}}{\partial x_0^{\al_0} \partial x_1^{\al_1} \partial x_2^{\al_2} \cdots \partial x_n^{\al_n} } 
\equiv \partial^{\al_0}_{x^0} \partial^{\al_1}_{x^1} \partial^{\al_2}_{x^2} \cdots \partial^{\al_n}_{x^n},
\nonumber
\end{gather}
where $|\al | = \al_0 +  \al_1 + \al_2 + \cdots + \al_n$.

\subsection{Gevrey spaces\label{section_appendix_Gevrey}}

In this section we review the definition of Gevrey spaces. Roughly speaking,
a function is of Gevrey class if it obeys inequalities similar, albeit weaker,
than those satisfied by analytic functions. One of the crucial properties of Gevrey spaces
for their use in general relativity is that they admit compactly supported functions.

\begin{definition} Let $s \geq 1$.
We say that $f:  \RR^n  \rar \CC$ belongs to the 
Gevrey space $G^{(s)}(\RR^n)$ if 
\begin{gather}
\sup_\al  \frac{1}{(1+|\al|)^s} \norm{\partial^\al f}_{L^2(\RR^n)}^\frac{1}{1+|\al|} 
< \infty.
\nonumber
\end{gather}
Let $K\subset \RR^n$ be the cube of unit side. We say that $f$ belongs to the local Gevrey space $G^{(s)}_{loc}(\RR^n)$ if 
\begin{gather}
\sup_\al  \frac{1}{(1+|\al|)^s} \left(\sup_K \norm{\partial^\al f}_{L^2(K)}\right)^\frac{1}{1+|\al |} 
< \infty,
\nonumber
\end{gather}
where $\sup_K$ is taken over all side one cubes $K$ in $\RR^n$.
\label{definition_Gevrey}
\end{definition}

We note that the case $s=1$, i.e., $G^{(1)}(\RR^n)$, corresponds to the space of analytic functions.

We next introduce the space of maps defined on $X$ whose derivatives up to order
$m$ belong to $G^{(s)}(\{x^0=t\})$, $0\leq t \leq T$.

\begin{definition}
On $X$, denote $S_t = \{ x^0 = t\}$. Let $s\geq 1$, and let $m\geq 0$ be an integer.
We denote by $\overline{\al}$ a multi-index $\al=(\al_0, \dots, \al_n)$ for which $\al_0 = 0$.
We define $G^{m,{(s)}}(X)$ as the set of maps $f:X \rar \CC$ such that
\begin{gather}
\sup_{\overline{\al}, \, |\beta| \leq m, \, 0 \leq t \leq T}
\frac{1}{( 1 + |\overline{\al} |)^s} \norm{\partial^{\be+\overline{\al}} f }_{L^2(S_t)}^\frac{1}{1+|\overline{\al}|} < \infty.
\nonumber
\end{gather}
Let $Y$ be an open set of $\RR^d$. We define $G^{m,{(s)}}(X\times Y)$ as the set of maps $f:X\times Y \rar \CC$ such that
\begin{gather}
\sup_{\overline{\al}, \, \ga, \, |\beta| \leq m, \, 0 \leq t \leq T}
\frac{1}{( 1 + |\overline{\al} | + |\ga| )^s} 
\norm{\sup_{y \in Y} \left| \partial_x^{\be+\overline{\al}} \partial_y^\ga f \right| }_{L^2(S_t)}^\frac{1}{1+|\overline{\al}| + |\ga|} < \infty.
\nonumber
\end{gather}
Let $K_t \subset S_t$ be the cube whose sides have unit length. 
The spaces $G^{m,{(s)}}_{loc}(X)$ and $G^{m,{(s)}}_{loc}(X\times Y)$ are defined as the set of maps $f:X \rar \CC$
and  $f:X\times Y \rar \CC$, respectively, such that
\begin{gather}
\sup_{\overline{\al}, \, |\beta| \leq m, \, 0 \leq t \leq T}
\frac{1}{( 1 + |\overline{\al} |)^s} \left(\sup_{K_t} \norm{\partial^{\be+\overline{\al}} f }_{L^2(K_t)} \right)^\frac{1}{1+|\overline{\al}|} < \infty,
\nonumber
\end{gather}
and
\begin{gather}
\sup_{\overline{\al}, \, \ga, \, |\beta| \leq m, \, 0 \leq t \leq T}
\frac{1}{( 1 + |\overline{\al} | + |\ga| )^s} 
\left( \sup_{K_t} \norm{\sup_{y \in Y} \left| \partial_x^{\be+\overline{\al}} \partial_y^\ga f \right| }_{L^2(K_t)} \right)^\frac{1}{1+|\overline{\al}| + |\ga|} < \infty,
\nonumber
\end{gather}
where $\sup_{K_t}$ is taken over all cubes of side one within $S_t$.
\label{definition_Gevrey_derivative}
\end{definition}

\begin{remark}
Definitions \ref{definition_Gevrey} and \ref{definition_Gevrey_derivative} are easily generalized
to vector and tensor fields in $\RR^n$ and $X$, and to open subsets of $\RR^n$ and $X$.
In particular, replacing $\RR^n$ by an open set $\Omega$ and $X$ by $[0,T]\times \Omega$
in the above definitions we obtain the corresponding spaces for $\Om$. This allows one to define
Gevrey spaces on manifolds. If $M$ is a differentiable manifold, we say that $f: M \rar \CC$
belongs to $G^{(s)}(M)$ if for every $p \in M$ there exists a coordinate chart $(x,U)$ about
$p$ such that $f\circ x^{-1} \in G^{(s)}(\Om)$, where $\Om = x(U)$.
This definition generalizes for vector and tensor fields.
\end{remark}

\begin{remark}
The reason to treat $X$ and $Y$ differently in definitions of $G^{(s)}(X\times Y)$ 
and $G^{m,{(s)}}(X\times Y)$ is that, in the theorems of section \ref{section_appendix_Cauchy},
 we need
to distinguish between the regularity with respect to the space-time $X$ and the regularity with respect to the parametrization of the initial data.
\end{remark}

\begin{remark}
We could similarly define for manifolds the analog of the other Gevrey spaces introduce above. 
However, this can be somewhat cumbersome and not always natural. In particular, the spaces
$G^{m,(s)}$ require a distinguished coordinate that plays the role of time. This can always be 
done locally, and it can in done for globally hyperbolic manifolds if we fix a particular 
foliation in terms of space-like slices (as done, e.g., in \cite{DisconziCzubakNonzero,DisconziViscousFluidsNonlinearity}), although it is debatable
how canonical this is. Here we prefer to avoid extra complications, i.e., we 
in fact only need the definition of $G^{(s)}(\Si)$, which is used
for the construction of appropriate local coordinates and the construction of the initial data
for the system in $\RR^4$ (sections \ref{section_initial_data} and \ref{section_initial_R4})
and in the results of section \ref{section_existence_causality_full}. The bulk of the proofs
are carried out for the system in $\RR^4$, where all the different Gevrey spaces play a role.
It follows that the solution in $\RR^4$ is in particular smooth, giving rise to a smooth
globally hyperbolic development. Note that for the conclusion
of Theorem \ref{main_theorem} it is not needed to assert that the full solution enjoys certain 
Gevrey regularity.
\end{remark}

For more about Gevrey spaces, see, e.g., \cite{LerayOhyaNonlinear,RodinoGevreyBook}.
We remark that the terminology ``local" and the notation $G_{loc}$ are not standard.

\subsection{The Cauchy problem\label{section_appendix_Cauchy}}

Let  $a = a(x,\partial^k)$, $x \in X$, be a linear differential operator of order $k$. We can write
\begin{gather}
a(x,\partial^k) = \sum_{ |\al | \leq k } a_\al(x) \partial^\al,
\nonumber
\end{gather}
where $\al = (\al_0, \al_1, \al_2, \dots, \al_n)$ is a multi-index.
Let $p(x,\partial^k)$ be the principal 
part of $a(x,\partial^k)$, i.e., 
\begin{gather}
p(x,\partial^k) = \sum_{ |\al | = k } a_\al(x) \partial^\al.
\nonumber
\end{gather}
At each point $x \in X$ and for each co-vector $\xi \in T_x^* X$, where $T^*X$ is the cotangent bundle of $X$,
we can associate a polynomial of order $k$ in the cotangent space
$T_x^* X$ obtained by replacing the derivatives by $\xi \in T_x^* X$. More precisely, for each 
$k^{\text{th}}$ order derivative in $a(x,\partial^k)$, i.e., 
\begin{gather}
\partial^\al 
=
\frac{\partial^{|\al |}}{\partial x_0^{\al_0} \partial x_1^{\al_1} \partial x_2^{\al_2} \cdots \partial x_n^{\al_n} } 
\nonumber
\end{gather}
$|\al | = k$,
we associate the polynomial
\begin{gather}
\xi^\al \equiv \xi_0^{\al_0} \xi_1^{\al_1} \xi_2^{\al_2} \cdots \xi_n^{\al_n},
\nonumber
\end{gather}
where $\xi = (\xi_0, \xi_1,\xi_2,\dots, \xi_n) \in T_x^* X$, forming in this way the polynomial
\begin{gather}
p(x,\xi) = \sum_{ |\al | = k } a_\al(x) \xi^\al.
\nonumber
\end{gather}
Clearly, $p(x,\xi)$ is a homogeneous polynomial of degree $k$. It is called the 
characteristic polynomial (at $x$) of the operator $a$.

The cone $V_x(p)$ of $p$
in $T_x^* X$ is defined by the equation
\begin{gather}
p(x,\xi) = 0.
\nonumber
\end{gather}
\begin{definition}
With the above notation, $p(x,\xi)$ is called a hyperbolic polynomial (at $x$) if there exists  $\zeta \in T_x^*X$ such 
that every straight line through $\zeta$ that does not contain the origin intersects 
the cone $V_x(p)$ at $k$ real distinct points.  The differential operator $a(x,\partial^k)$
is called a hyperbolic operator (at $x$) if $p(x,\xi)$ is hyperbolic.
\label{definition_hyperbolic_polynomial}
\end{definition}

Leray proved in
 \cite{Leray_book_hyperbolic}
 that  (if $X$ is at least three-dimensional)  if $p(x,\xi)$ is hyperbolic at $x$, then 
 the set of points $\zeta$ satisfying the condition of Definition \ref{definition_hyperbolic_polynomial}
forms the interior
of two opposite  half-cones $\Ga_x^{*,+}(a)$, $\Ga_x^{*,-}(a)$, with
$\Ga_x^{*,\pm}(a)$ non-empty, with boundaries that belong to 
$V_x(p)$ .

\begin{remark}
Another way of stating Definition \ref{definition_hyperbolic_polynomial} is as follows. Given
$\zeta \in T_x X$, consider a non-zero vector $\theta$ that is not parallel to $\zeta$ and
form the line $\lambda \zeta + \theta$, where $\lambda \in \RR$ is a parameter. We then require
this line to intersect the cone $V_x(p)$ at $k$ distinct real points.
An equivalent definition of hyperbolic polynomials is as follows \cite{Courant_and_Hilbert_book_2}:
$p(x,\xi)$ is hyperbolic at $x$ if for each non-zero $\xi=(\xi_0,\dots, \xi_n) \in T^*_x X$, the equation
$p(x,\xi)=0$ has $k$ distinct real roots $\xi_0 = \xi_0(\xi_1,\dots,\xi_n)$.
\end{remark}

With applications to systems in mind, we next consider the $N \times N$ diagonal linear differential operator matrix
\begin{gather}
A(x,\partial) = 
\left(
\begin{matrix}
a^1(x,\partial^{k_1}) & \cdots & 0 \\
\vdots & \ddots & \vdots \\
0 & \cdots & a^N(x,\partial^{k_N})
\end{matrix}
\right).
\nonumber
\end{gather}
Each $a^J(x,\partial^{k_J})$, $J = 1, \dots, N$ is a linear differential operator
of order $k_J$. 

\begin{definition}
The operator  $A(x,\partial)$ is called 
Leray-Ohya  hyperbolic (at $x$) if:

(i) The characteristic polynomial $p^J(x,\xi)$ of each $a^J(x,\partial^{k_J})$ is a product 
of  hyperbolic
polynomials, i.e.
\begin{gather}
p^J(x,\xi) = p^{J,1}(x, \xi) \cdots p^{J,r_J}(x,\xi),\, J=1,\dots, N,
\nonumber
\end{gather}
where each $p^{J,q}(x,\xi)$, $q=1,\dots,r_J$, $J=1,\dots,N$,
is a hyperbolic polynomial.

(ii) The two opposite
convex half-cones,
\begin{gather}
\Ga_x^{*,+}(A) = \bigcap_{J=1}^N \bigcap_{q=1}^{r_J} \Ga_x^{*,+}(a^{J,q}),
\, \text{ and } \,
\Ga_x^{*,-}(A) = \bigcap_{J=1}^N \bigcap_{q=1}^{r_J} \Ga_x^{*,-}(a^{J,q}),
\nonumber
\end{gather}
have a non-empty interior. Here, $\Ga_x^{*,\pm}(a^{J,q})$ are the
half-cones associated with the hyperbolic polynomials $p^{J,q}(x,\xi)$,
$q=1,\dots,r_J$,
$J=1,\dots,N$.
\label{definition_Leray_Ohya_hyperbolic}
\end{definition}

\begin{remark}
When the above hyperbolicity properties hold for every $x$, we call the corresponding
operators hyperbolic (we can also talk about hyperbolicity in an open set, a certain region, etc.).
When we say that an operator is Leray-Ohya hyperbolic on the whole space (or in an open set, etc.),
this means not only that Definition \ref{definition_Leray_Ohya_hyperbolic} applies for every
$x$, but also that the numbers $r_J$ and the degree of the polynomials 
$p^{J,q}(x,\xi)$, $q=1,\dots,r_J$, $J=1,\dots,N$, do not change with $x$.
\end{remark}

\begin{definition}
We define the dual convex half-cone $C_x^+(A)$ at $T_x X$ as the
set of $v \in T_x X$ such that $\xi(v) \geq 0$ for every
$\xi \in \Ga_x^{*,+}(A)$; $C_x^-(A)$ is analogously defined, and 
we set  $C_x(A) = C_x^+(A) \cup C_x^-(A)$. 
If the convex cones $C_x^+(A)$ and
$C_x^-(A)$ can be continuously distinguished with respect to $x \in X$,
then $X$ is called time-oriented (with respect to the hyperbolic form
provided by the operator $A$). A path in $X$ is called
future (past) time-like with respect to $A$ if its tangent at each point 
$x \in X$ belongs to $C_x^+(A)$ ($C_x^-(A)$), and future (past) 
causal if its tangent at each point  $x \in X$ belongs or is tangent 
to $C_x^+(A)$ ($C_x^-(A)$). A regular surface $\Si$ is 
called space-like with respect to $A$ if  $T_x\Si$ ($ \subset T_x X$) is
exterior to $C_x(A)$ for each $x \in \Si$.  It follows that for a time-oriented
$X$,
 the concepts of causal past, 
future, domains of dependence and influence of a set can be defined 
in the same way one does when the manifold is endowed with a Lorentzian 
metric. We refer the reader to \cite{Leray_book_hyperbolic} for details. Here we need only 
the following: the causal past $J^-(x)$ of a point $x \in X$ is the set of points that can
be joined to $x$ by a past causal curve. 
\label{definition_causal_A}
\end{definition}

\begin{remark}
The definitions in Definition \ref{definition_causal_A} endow $X$ with a
causal structure provided by 
the operator $A$.
Despite the similar terminology, however, it should be noticed that all of the above
definitions depend only on the structure of the operator $A$, and do not require
an a priori  Lorentzian metric on $X$. The case of interest in general relativity, however, is when
the causal structure of the space-time is connected with that of $A$. In this regard, the following observation is useful.
Suppose that $X$ has a Lorentzian metric $g$.
For causal solutions 
of the systems of equations here described (see Theorem \ref{Leray_Ohya_causality} below) to be causal in the sense 
of general relativity, one needs that, for all $x \in X$, $C_x^\pm(A) \subseteq K_x^\pm$, where 
$K_x^\pm$ are the two halves of the light-cone $\{ g_{\mu \nu} \xi^\mu \xi^\nu \leq 0 \}$.
By duality, this is equivalent to saying that in the cotangent spaces we have $K_x^{*,\pm} \subseteq \Ga_x^{*,+}(A) $,
where  $K_x^{*,\pm}$ are the two halves of the dual light-cone $\{ g^{\mu \nu} \xi_\mu \xi_\nu \leq 0 \}$.
\label{remark_causality_A_g}
\end{remark}

Next, we consider the following quasi-linear system of differential 
equations
\begin{gather}
A(x,U,\partial) U = B(x, U),
\label{quasi_linear_system}
\end{gather}
where $A(x,U, \partial)$ is the $N \times N$ diagonal matrix
\begin{gather}
A(x,U, \partial) = 
\left(
\begin{matrix}
a^1(x,U, \partial^{k_1}) & \cdots & 0 \\
\vdots & \ddots & \vdots \\
0 & \cdots & a^N(x,U, \partial^{k_N})
\end{matrix}
\right),
\nonumber
\end{gather}
with $a^J(x,U,\partial^{k_J})$, $J = 1, \dots, N$ differential operators
of order $k_J$. $B(x,U)$ is the vector
\begin{gather}
B(x,U) = (b^J(x,U)),\, J=1,\dots, N,
\nonumber
\end{gather}
and the vector 
\begin{gather}
U(x) = (U^I(x)), \, I = 1, \dots, N
\nonumber
\end{gather}
is the unknown. Notice that because $a_J$ is allowed to depend on $U$, the above system is in general 
non-linear.

\begin{definition}
The system $A(x,U,\partial) U = B(x, U)$ is called a 
Leray system
if it is possible to attach to each unknown $u^I$ an integer $m_I \geq 0$,
and to each equation $J$ of the system an integer $n_J \geq 0$,
such that:

(i) $k_J = m_J - n_J$, $J=1,\dots, N$;

(ii) the functions $b^J$ and the coefficients of the differential operators
$a^J$ are\footnote{The regularity required for the coefficients
$a^J$ and $b^J$ depends on particular applications and context. For instance,
for Theorem \ref{Leray_Ohya_theorem_existence} the required regularity
is specified. Similarly, in Definition \ref{definition_Cauchy_Leray}, one needs
to take derivatives of these quantities up to order $n_J$, thus they need to be
at least as many times differentiable.} functions of $x$, of $u^I$, and of the 
derivatives of $u^I$ of order at most $m_I - n_J - 1$, 
$I, J =1\dots, N$. If for some $I$ and some $J$, $m_I - n_J< 0$, then
the corresponding $a^J$ and $b^J$ do not depend on $u^I$.
\label{definition_Leray_system}
\end{definition}

\begin{remark} The indices $m_I$ and $n_J$ in Definition \ref{definition_Leray_system}
are defined up to an additive integer.
\end{remark}

\begin{definition}
A Leray-Ohya system (with diagonal principal part) is a Leray system where the matrix $A$ is Leray-Ohya
hyperbolic. In the quasi-linear case, since the operators $a$ depend on $U$,
we need to specify a function 
$U$ that  is  plugged into $A(x,U,\partial)$ in order to compute the characteristic polynomials. 
In this case we talk about
a Leray-Ohya system for the function $U$. The primary case of interest
is when $U$ assumes the values of the given Cauchy data.
\label{definition_Leray_Ohya_system}
\end{definition}

When considering a quasi-linear system, we write $p(x,U,\xi)$ and similar expressions to indicate 
the dependence on $U$. 
 
We now formulate the Cauchy problem for Leray systems.

\begin{definition}
Let $\Si$ be a regular hypersurface in $X$, which we assume for simplicity to be given by $\{ x^0 = 0 \}$. 
The Cauchy data on $\Si$ for a Leray 
system in $X$ consists of the values
of $U=(u^I)$ and their derivatives up to order $m_I - 1$ on $\Si$,
i.e., $\left. \partial^\al u^I \right|_{\Si}$, $|\al| \leq m_I -1$, $I=1,\dots, N$. 
The Cauchy data is required to satisfy the following compatibility conditions. If
$V=(v^I)$ is an extension of the Cauchy data defined in a neighborhood  
of $\Si$, i.e.
$\left. \partial^\al v^I \right|_{\Si} = \left. \partial^\al u^I \right|_{\Si}$, $|\al| \leq m_I -1$,
 $I=1,\dots, N$, 
then the difference $a^J(x,V, \partial)U - b^J(x,V)$
and its derivatives of order less than $n_J$ vanish on $\Si$, for
$J=1,\ldots, N$. When to a Leray system
 $A(x,U, \partial) U = B(x, U)$
we prescribe initial data satisfying these conditions, we say that we have 
a Cauchy problem for $A(x,U, \partial) U = B(x, U)$. 
\label{definition_Cauchy_Leray}
\end{definition}

Notice that by definition,
the Cauchy data for a Leray system satisfies the aforementioned compatibility
conditions. We also introduce the following notions related to the
Cauchy problem for a Leray system. 

\begin{assumption}
Consider the Cauchy problem for a Leray system $A(x,U,\partial) U = B(x,U)$. 
Let $Y$ be an open set of $\RR^L$, where $L$ equals the number of derivatives 
of $u^J$ of order less or equal to $\max_{I} m_I - n_J$,  $J=1,\dots,N$, and
such that $Y$ contains the closure of the values taken by the Cauchy data on $\Si$.
It is convenient to consider  $A(x,U,\partial)$ as a differential operator defined over $X \times Y$, 
as follows. We shall assume
that there exists a differential operator $\widetilde{A}(x,y,\partial)$
defined over $X\times Y$ with the following property.  If $(x,y) \in X\times Y$ and 
$V = (v^J)$ is a sufficiently regular
function on $\Si$ such that $y = (\partial^{\max_{I} m_I - n_J } v^J (x))_{J=1,\dots,N}$, 
then $A(x,V(x),\partial) = \widetilde{A}(x,y,\partial)$. We shall write $A(x,y,\partial)$ for 
$\widetilde{A}(x,y,\partial)$.
\label{assumption_structure_operators}
\end{assumption}

\begin{definition}
Consider the Cauchy problem for a Leray system $A(x,U,\partial) U = B(x,U)$.
Let $\Si$ and $Y$ be as in Definition \ref{definition_Cauchy_Leray} and
Assumption \ref{assumption_structure_operators}, respectively.
Denote by $\cA^s(\Si,I)$ the set of $V = (v^J) \in G^{(s)}(\Si)$, $J=1,\dots,N$,
such that $ (\partial^{\max_{I} m_I - n_J } v^J (x))_{J=1,\dots,N} \in Y$ for all $x \in \Si$.
\label{definition_A_s}
\end{definition}

We are now ready to state the results of this appendix.
We use the above notation and definitions in the statement of the theorems below. 
 
\begin{theorem} (Existence and uniqueness)
Consider the Cauchy problem for (\ref{quasi_linear_system}).  Suppose that the Cauchy data 
is in $G^{(s)}(\Si)$, and that 
\begin{gather}
a^J( \cdot,\cdot, \partial^{k_J}) \in G_{loc}^{n_J,(s)}(X \times Y), \, \text{ and } \, b^J(\cdot,\cdot) \in G^{n_J,(s)}(X \times Y).
\nonumber
\end{gather}
Suppose that for any $V \in \cA^s(\Si,Y)$
the system is Leray-Ohya hyperbolic with indices $m_I$ and $n_J$;
thus for all $x \in \Si$, each $p^J(x,V, \xi)$ is the product of 
$r_J$ hyperbolic polynomials,
\begin{gather}
p^J(x,V,\xi) = p^{J,1}(x, V,\xi) \cdots p^{J, r_J}(x, V, \xi),\, J=1,\dots, N.
\nonumber
\end{gather}
Suppose that each $p^{J,q+1}(x, V, \xi)$, $q=0, \dots, r_J-1$, depends on 
at most $m_I - m_{J,q} - r_I +q$ derivatives of $v^I$, $I=1,\dots,N$, where
\begin{gather}
m_{J,q} = n_J + \operatorname{deg}(p^{J,1}) + \cdots +\operatorname{deg}(p^{J,q}), \, m_{J, r_J} = m_J, \, m_{J,0} = n_J. 
\nonumber
\end{gather}
Above, $\operatorname{deg}(p^{J,q})$ is the degree, in $\xi$, of the polynomial 
$p^{J,q}(x, V,\xi)$. 

Denote by $a^J_{q+1}(x,y,\partial)$ the differential operator associated with 
$p^{J,q+1}$. Assume that
\begin{gather}
a^J_{q+1}(\cdot,\cdot,\partial) \in G_{loc}^{m_{J,q} - q, (s)}(X\times Y).
\nonumber
\end{gather}
Let $0 \leq g_I \leq r_I$ be the smallest integers such that $a^J(x,V,\partial^{m_J-n_J})$
and $b^J(x,V)$ depend on at most $m_I - n_J - r_I + g_I$ derivatives of 
$v^I$, $I=1,\dots, N$, $J=1,\dots, N$.
Finally, assume that 
\begin{gather}
1 \leq s \leq \frac{r_J}{g_J} \, \text{ and } \, \frac{n}{2} + r^J < n_J, \, J=1,\dots, N.
\nonumber
\end{gather}
Then, there exists a $T^\prime > 0$ and a solution $U=(u^I)$ to the Cauchy problem for (\ref{quasi_linear_system}) and defined on $[0,T^\prime) \times \RR^n \subseteq X$. The solution satisfies
\begin{gather}
u^I \in G^{m_I,(s)}([0,T^\prime) \times \RR^n), \, I=1,\dots, N.
\nonumber
\end{gather}
Furthermore, the solution is unique in this regularity class.
\label{Leray_Ohya_theorem_existence}
\end{theorem}

\begin{theorem} (Causality) Assume the same hypotheses
of Theorem \ref{Leray_Ohya_theorem_existence}, and suppose further that 
\begin{gather}
1 \leq s < \frac{r_J}{g_J}, \, J=1,\dots, N.
\nonumber
\end{gather}
Let $T^\prime$ and $U$ be as in the conclusion of  Theorem \ref{Leray_Ohya_theorem_existence}. 
Then, if $T^\prime$ is sufficiently small, the operator $A(x,U,\partial)$
is Leray-Ohya hyperbolic (thus the causal past of a point is well-defined),
and for each $x \in [0,T^\prime) \times \RR^n$, $U(x)$ depends only on
$\left. U_0 \right|_{J^-(x) \cap \Si}$, where $U_0$ is the Cauchy data.
\label{Leray_Ohya_causality}
\end{theorem}

\begin{remark}
Theorem \ref{Leray_Ohya_theorem_existence} assumes that the system is Leray-Ohya hyperbolic
for $V \in \cA(\Si,Y)$, which is essentially the space of values near the initial data.
(Naturally, it would not make sense to require the system to be Leray-Ohya hyperbolic
for the yet to be proven to exist solution $U$.) Once $U$ is constructed, one can then
ask whether the system is Leray-Ohya hyperbolic for $U$. This will be the case if $T^\prime$
is small, since in this case the values of $U$ will be close to those of the initial data
by continuity, guaranteeing that $U(x) \in \cA(\Si,Y)$.
\end{remark}

Theorems \ref{Leray_Ohya_theorem_existence} and \ref{Leray_Ohya_causality}
are proven in \cite{LerayOhyaNonlinear} (reprinted in \cite{LerayOhyaNonlinearReprint}).

We now consider a system whose principal part is not necessarily diagonal.
The definition of a Leray system depends only on the existence of the indices $m_I$ and
$n_J$ with the stated properties, and thus can be extended to non-diagonal systems.

\begin{definition}
Consider a system of $N$ partial differential 
equations and $N$ unknowns in $X$,
and
denote the unknown as 
$U=(u^I)$, $I=1,\dots, N$. 
The system is a (not necessarily diagonal in the principal part) Leray system if
it is possible 
to attach to each unknown $u^I$ a non-negative integer $m_I$ and to
each equation a non-negative integer $n_J$, such that the system reads
\begin{gather}
h^J_I(x,\partial^{m_K - n_J -1} u^K, \partial^{m_I - n_J}) u^I
+ b^J(x, \partial^{m_K - n_J - 1} u^K) = 0, \, J=1, \dots, N.
\label{general_system}
\end{gather}
Here, 
$h^J_I(x,\partial^{m_K - n_J -1} u^K, \partial^{m_I - n_J})$
is a homogeneous differential operator of order $m_I - n_J$ (which can
be zero), whose coefficients depend on at most
$m_K - n_J -1$ derivatives of $u^K$, $K=1,\dots N$, and there is a sum over $I$ in $h^J_I(\cdot) u^I$. 
The remaining terms, 
$b^J(x,\partial^{m_K - n_J - 1} u^K)$, also depend on at most 
$m_K - n_J -1$ derivatives of $u^K$,  $K=1,\dots N$. 
As before, these indices are defined only up to an
overall additive integer.
\end{definition}

As done above, for a given sufficiently regular $U$,   
 $h^J_I(x,\partial^{m_K - n_J -1} U^K, \partial^{m_I - n_J})$  are well-defined linear
 operators, and we can ask about their hyperbolicity properties. The case of interest will be, again, when
 we evaluate these operators at some given Cauchy data.
 
Write (\ref{general_system}) in matrix form as
 \begin{gather}
 H(x,U,\partial) U = B(x,U).
 \label{general_system_matrix}
\end{gather}
 
\begin{definition} 
The characteristic determinant of (\ref{general_system_matrix}) 
at $x \in X$ and for a given $U$
 is the polynomial $p(x,\xi)$ in the co-tangent space $T^*_x X$, $\xi \in T_x^* X$, 
given by 
\begin{gather}
p(x,U,\xi) = \det( H(x,U,\xi)).
\label{characteristic_det_appendix}
\end{gather}
\end{definition}
Note that $p$ is a homogeneous polynomial of degree
\begin{gather}
 \ell \equiv \sum_{I=1}^N m_I - \sum_{J=1}^N n_J.
\nonumber
\end{gather}
Under appropriate conditions, (\ref{general_system_matrix}) can be transformed 
into a Leray-Ohya system of the form (\ref{quasi_linear_system}), i.e., with diagonal principal part.
More precisely, we have the following.

\begin{theorem} (Diagonalization)
Consider (\ref{general_system_matrix}).
Suppose that the characteristic determinant (\ref{characteristic_det_appendix}) at a given $U$ is not identically zero, and it is the product
of $Q$ hyperbolic polynomials, i.e., 
\begin{gather}
p(x,U,\xi) = p_1(x,U,\xi) \cdots p_Q(x,U,\xi).
\nonumber
\end{gather}
Let $d_q$ be the degree of $p_q(x,U,\xi)$, $q=1,\dots, Q$, and suppose
that 
\begin{gather}
\max_q d_q \geq \max_I m_I - \min_J n_J.
\nonumber
\end{gather}
Finally, assume that 
\begin{gather}
\ell \geq \max_I m_I - \min_J n_J.
\nonumber
\end{gather}
Then, there exists a $N\times N$ matrix $C(x,U,\partial)$ of differential operators whose coefficients depend on $U$, such that
\begin{gather}
C(x,U,\partial) H(x,U,\partial) U =\mathbb{I} \, p(x,U,\partial) U + \widetilde{B}_1(x,U),
\nonumber
\end{gather}
and 
\begin{gather}
C(x,U,\partial) B(x,U) = \widetilde{B}_2(x,U),
\nonumber
\end{gather}
where $\mathbb{I}$ is the $N\times N$ identity matrix, $p(x,U,\partial)$ is the differential operator associated with 
$p(x,U,\xi)$, and $\widetilde{B}_1(x,U)$ and $ \widetilde{B}_2(x,U)$ depend on at most $\ell -1 $ derivatives of $U$, as do the coefficients
of the operator $p(x,U,\xi)$. Furthermore, there is a choice of indices that makes the system
\begin{gather}
\mathbb{I} \, p(x,U,\partial) U =  \widetilde{B}_2(x,U) - \widetilde{B}_1(x,U)
\label{diagonalized_Leray}
\end{gather}
into a Leray system. In particular,
if the intersections $\cap_q \Ga_x^{*,+}(a^q)$ and $\cap_q \Ga_x^{*,-}(a^q)$, where
$\Ga_x^{*,\pm}(a^q)$ are the half-cones associated with the hyperbolic polynomials $p_q(x,U,\xi)$,
have non-empty interiors, then 
 (\ref{diagonalized_Leray}) is a Leray-Ohya system with diagonal principal part 
in the sense of definition \ref{definition_Leray_Ohya_system}.
\label{Leray_Ohya_diagonal}
\end{theorem}

Theorem \ref{Leray_Ohya_diagonal} is proven in \cite{CB_diagonal}.

\begin{definition}
Under the hypotheses of Theorem \ref{Leray_Ohya_diagonal}, the number $\frac{Q}{Q-1}$ is 
called the Gevrey index of the system.  
\end{definition}

\begin{remark}
Suppose that (\ref{diagonalized_Leray}) forms a Leray-Ohya system in the sense
of definition \ref{definition_Leray_Ohya_system}, i.e., the half-cones have non-empty interiors
as stated in Theorem \ref{Leray_Ohya_diagonal}. 
It can then be shown \cite{CB_diagonal} that a value of $s$ sufficient to apply
Theorems \ref{Leray_Ohya_theorem_existence} and \ref{Leray_Ohya_causality}
is $1 \leq s < \frac{Q}{Q-1}$.
\label{remark_Gevrey_index}
\end{remark}

Let us make a brief comment about the proofs of the above results. Theorem \ref{Leray_Ohya_theorem_existence}
is proven as follows. First, one solves the associated linear problem. This is done by a method of majorants reminiscent 
of the Cauchy-Kowalevskaya theorem. One uses the fact that Gevrey functions admit a formal series expansion that provides
a consistent way of constructing successive approximating solutions to the problem. The non-linear problem is then treated
via a fixed point argument, upon solving successive linear problems. Theorem \ref{Leray_Ohya_causality} is obtained
by a Holmgren type of argument. We remark that the assumption that 
$p^{J,q+1}(x,V, \xi)$, $q=0, \dots, r_J-1$, depends on 
at most $m_I - m_{J,q} - r_I +q$ derivatives of $v^I$, $I=1,\dots,N$,
ensures that the coefficients of the associated differential operators 
$a^{J,q+1}(x,U, \partial)$ do not depend on too many derivatives of $U$,
as it should be in the treatment of quasi-linear equations.

Theorem \ref{Leray_Ohya_diagonal} is based on the following identity:
\begin{gather}
c^T a = \det(a),
\end{gather}
where $a$ is an $N\times N$ invertible matrix and $c^T$ the transpose of the co-factor matrix. At the level of differential operators,
this identity produces the lower order terms $\widetilde{B}_1$. One then needs to match the order of the resulting differential operators
and lower order terms with appropriate indices satisfying the definition of a Leray system. This is possible under the conditions
on $d_q$ and $\ell$ stated in the theorem.

\section{Derivation of the equations of motion}
\label{appendix_derivation}

In this section we give the derivation of 
(\ref{EE_gauge}) and 
(\ref{div_T_gauge}). The derivation of (\ref{EE_gauge}) is standard and we include
it here for the reader's convenience, thus let us start with (\ref{EE_gauge}).
Let 
\begin{gather}
{}^{(0)}t_{\al\be} = \frac{4}{3} u_\al u_\be +\frac{1}{3} g_{\al\be} \ep,
\label{ideal_tensor}
\end{gather}
and denote the third to ninth terms in \eqref{conformal_tensor} by 
${}^{(1)}t_{\al\be}$ to ${}^{(7)}t_{\al\be}$, respectively. 
Explicitly,
\begin{align}
\begin{split}
{}^{(1)}t_{\al\be} &= -\eta \pi_\al^\mu \pi_\be^\nu (\nabla_\mu u_\nu + \nabla_\nu u_\mu 
-\frac{2}{3} g_{\mu\nu} \nabla_\la u^\la ),
\\
{}^{(2)}t_{\al\be} & = \la (u_\al u^\mu \nabla_\mu u_\be + u_\be u^\mu \nabla_\mu u_\al) ,
\\
{}^{(3)}t_{\al\be} & = \frac{1}{3} \chi \pi_{\al\be} \nabla_\mu u^\mu,
\\
{}^{(4)}t_{\al\be} &=   \chi u_\al u_\be \nabla_\mu u^\mu ,
\\
{}^{(5)}t_{\al\be}&= \frac{\la}{4 \ep }(u_\al \pi^\mu_\be \nabla_\mu \ep + u_\be \pi^\mu_\al \nabla_\mu \ep) ,
\\
{}^{(6)}t_{\al\be} & = 
\frac{3 \chi}{4 \ep} u_\al u_\be u^\mu\nabla_\mu \ep,
\\
{}^{(7)}t_{\al\be} & =   \frac{\chi}{4 \ep}
\pi_{\al\be} u^\mu \nabla_\mu \ep,
\end{split}
\nonumber
\end{align}
so that 
\begin{gather}
T_{\al\be} = {}^{(0)}t_{\al\be} + {}^{(1)}t_{\al\be} + \cdots {}^{(7)}t_{\al\be}.
\nonumber
\end{gather}

\subsection{Calculation of $\nabla_\al {}^{(1)}t^\al_\be$}

We have 
\begin{align}
\begin{split}
\nabla_\al {}^{(1)}t^\al_\be & = -\eta \pi^{\al \mu} \pi_\be^\nu(\nabla_\al \nabla_\mu u_\nu
+ \nabla_\al \nabla_\nu u_\mu  -\frac{2}{3} g_{\mu\nu} \nabla_\al \nabla_\la u^\la )
\\ &
+ \nabla_\al (\eta \pi^{\al\mu} \pi^\nu_\be )(\nabla_\mu u_\nu + \nabla_\nu u_\mu 
-\frac{2}{3} g_{\mu\nu} \nabla_\la u^\la ).
\end{split}
\label{1.1}
\end{align}
Compute
\begin{align}
\begin{split}
\pi_\be^\nu\nabla_\al \nabla_\mu u_\nu & = (g^\nu_\be + u_\be u^\nu) \nabla_\al \nabla_\mu u_\nu 
= \nabla_\al \nabla_\mu u_\be + u_\be u^\nu \nabla_\al \nabla_\mu u_\nu
\\ &
=  \nabla_\al \nabla_\mu u_\be + u_\be \nabla_\al (u^\nu \nabla_\mu u_\nu)
- u_\be \nabla_\al u^\nu \nabla_\mu u_\nu 
\\ &
= \nabla_\al \nabla_\mu u_\be 
- u_\be \nabla_\al u^\nu \nabla_\mu u_\nu ,
\end{split}
\nonumber
\end{align}
so that 
\begin{align}
\begin{split}
-\eta \pi^{\al\mu} \pi_\be^\nu \nabla_\al \nabla_\mu u_\nu & = - \eta \pi^{\al\mu} 
(\nabla_\al \nabla_\mu u_\be -\nabla_\al u^\nu \nabla_\mu u_\nu )
\\ &
= -\eta (g^{\al \mu} + u^\al u^\mu ) \nabla_\al \nabla_\mu u_\be 
+\eta \pi^{\al\mu}\nabla_\al u^\nu \nabla_\mu u_\nu 
\\ & 
= -\eta g^{\al \mu} \nabla_\al \nabla_\mu u_\be + u^\al u^\mu  \nabla_\al \nabla_\mu u_\be 
+\eta \pi^{\al\mu}\nabla_\al u^\nu \nabla_\mu u_\nu.
\end{split}
\label{2.1}
\end{align}
Similarly, we find
\begin{align}
\begin{split}
\pi^{\al\mu} \nabla_\al \nabla_\nu u_\mu & = 
(g^{\al\mu} + u^\al u^\mu)\nabla_\al \nabla_\nu u_\mu = 
g^{\al\mu} \nabla_\al \nabla_\nu u_\mu + u^\al u^\mu \nabla_\al \nabla_\nu u_\mu
\\ &
=\nabla_\al \nabla_\nu u^\al - u^\al \nabla_\al u^\mu \nabla_\nu u_\mu,
\end{split}
\nonumber
\end{align}
so that 
\begin{align}
\begin{split}
-\eta \pi^{\al\mu} \pi^\nu_\be \nabla_\al \nabla_\nu u_\mu 
& = -\eta \pi_\be^\nu (\nabla_\al \nabla_\nu u^\al - u^\al \nabla_\al u^\mu \nabla_\nu u_\mu)
\\ &
= -\eta g^\nu_\be \nabla_\al \nabla_\nu u^\al  - \eta u_\be u^\nu \nabla_\al \nabla_\nu u^\al 
+\eta \pi^\nu_\be u^\al \nabla_\al u^\mu \nabla_\nu u_\mu .
\end{split}
\label{2.2}
\end{align}
But 
\begin{align}
\begin{split}
\nabla_\al \nabla_\nu u^\al = \nabla_\nu \nabla_\al u^\al + R_{\nu\al } u^\al,
\end{split}
\nonumber
\end{align}
so that \eqref{2.2} becomes
\begin{align}
\begin{split}
-\eta \pi^{\al\mu} \pi^\nu_\be \nabla_\al \nabla_\nu u_\mu 
& = 
-\eta g^\nu_\be (\nabla_\nu \nabla_\al u^\al + R_{\nu\al } u^\al) - \eta u_\be u^\nu 
(\nabla_\nu \nabla_\al u^\al + R_{\nu\al } u^\al)
\\ &
+\eta \pi^\nu_\be u^\al \nabla_\al u^\mu \nabla_\nu u_\mu
\\ &
= 
-\eta g^\nu_\be \nabla_\nu \nabla_\al u^\al 
-\eta g^\nu_\be R_{\nu\al } u^\al 
- \eta u_\be u^\nu \nabla_\nu \nabla_\al u^\al 
\\ &
- \eta u_\be u^\nu  R_{\nu\al } u^\al
+\eta \pi^\nu_\be u^\al \nabla_\al u^\mu \nabla_\nu u_\mu.
\end{split}
\label{3.1}
\end{align}
Next compute
\begin{align}
\begin{split}
-\eta \pi^{\al\mu} \pi_\be^\nu (-\frac{2}{3} g_{\mu\nu} \nabla_\al \nabla_\la u^\la  )
& =
\frac{2}{3}\eta \pi^{\al\mu} \pi_{\be\mu} \nabla_\al \nabla_\la u^\la  
\\
&
= \frac{2}{3}\eta \pi_\be^\al   \nabla_\al \nabla_\la u^\la  
=  \frac{2}{3}\eta (g^\al_\be + u^\al u_\be)  \nabla_\al \nabla_\la u^\la  
\\ &
=  \frac{2}{3} \eta  g^\al_\be \nabla_\al \nabla_\la u^\la  + 
\frac{2}{3} \eta u_\be u^\al  \nabla_\al \nabla_\la u^\la  .
\end{split}
\label{3.2}
\end{align}
Plugging \eqref{2.1}, \eqref{3.1}, and \eqref{3.2} into \eqref{1.1} we find
\begin{align}
\begin{split}
\nabla_\al {}^{(1)} t^\al_\be & = 
-\eta g^{\al\mu} \nabla_\al \nabla_\mu u_\be 
-\eta u^\al u^\mu \nabla_\al \nabla_\mu u_\be 
+ \eta u_\be \pi^{\al\mu} \nabla_\al u_\nu \nabla_\mu u^\nu 
- \eta g^\nu_\be \nabla_\nu \nabla_\al u^\al 
\\ &
-\eta u_\be u^\nu \nabla_\nu \nabla_\al u^\al 
- \eta R_{\be\al} u^\al 
-\eta u_\be R_{\nu\al} u^\nu u^\al 
+ \eta \pi_\be^\nu u^\al \nabla_\al u^\mu \nabla_\nu u_\mu 
\\ &
+\frac{2}{3} \eta g^\nu_\be \nabla_\nu \nabla_\al u^\al 
+\frac{2}{3} \eta u_\be u^\nu \nabla_\nu \nabla_\al u^\al 
\\
&
+ \nabla_\al (\eta \pi^{\al\mu} \pi^\nu_\be )(\nabla_\mu u_\nu + \nabla_\nu u_\mu 
- \frac{2}{3} g_{\mu\nu} \nabla_\la u^\la ).
\end{split}
\nonumber
\end{align}
We now group the first two terms, the fourth term with the ninth term, and the fifth term 
with the tenth term, to find
\begin{align}
\begin{split}
\nabla_\al {}^{(1)} t^\al_\be & = 
-\eta( g^{\al\mu} + u^\al u^\be ) \nabla_\al \nabla_\mu u_\be
-\frac{1}{3} \eta g^\nu_\be \nabla_\nu \nabla_\al u^\al 
\\
&
-\frac{1}{3} \eta u_\be u^\nu \nabla_\nu \nabla_\al u^\al 
+ {}^{(1)}B_\be,
\end{split}
\label{4.1}
\end{align}
where 
\begin{align}
\begin{split}
{}^{(1)}B_\be & = 
 \eta u_\be \pi^{\al\mu} \nabla_\al u_\nu \nabla_\mu u^\nu 
- \eta R_{\be\al} u^\al 
-\eta u_\be R_{\nu\al} u^\nu u^\al 
+ \eta \pi_\be^\nu u^\al \nabla_\al u^\mu \nabla_\nu u_\mu 
\\ &
+ \nabla_\al (\eta \pi^{\al\mu} \pi^\nu_\be )(\nabla_\mu u_\nu + \nabla_\nu u_\mu 
- \frac{2}{3} g_{\mu\nu} \nabla_\la u^\la ).
\end{split}
\label{4.2}
\end{align}

\subsection{Calculation of $\nabla_\al {}^{(2)}t^\al_\be$} Compute
\begin{align}
\begin{split}
\nabla_\al {}^{(2)}t^\al_\be & = \nabla_\al \big[ \la ( u^\al u^\mu \nabla_\mu u_\be 
+ u_\be u^\mu \nabla_\mu u^\al ) \big ]
\\ &
 = \la (u^\al u^\mu \nabla_\al \nabla_\mu u_\be + u_\be u^\mu \nabla_\al \nabla_\mu u^\al )
 + \nabla_\al (\la u^\al u^\mu) \nabla_\mu u_\be 
 \\
 &
 + \nabla_\al (\la u_\be u^\mu) \nabla_\mu u^\al .
\end{split}
\nonumber
\end{align}
Using
 $\nabla_\al \nabla_\mu u^\al = \nabla_\mu \nabla_\al u^\al + R_{\mu\al } u^\al$
we find
\begin{align}
\begin{split}
\nabla_\al {}^{(2)}t^\al_\be & = 
\la u^\al u^\mu \nabla_\al \nabla_\mu u_\be 
+ \la u_\be u^\mu \nabla_\mu \nabla_\al u^\al 
+ {}^{(2)}B_\be ,
\end{split}
\label{5.1}
\end{align}
where
\begin{align}
\begin{split}
 {}^{(2)}B_\be & = 
 \la u_\be R_{\mu \al } u^\mu u^\al 
 + \nabla_\al (\la u^\al u^\mu) \nabla_\mu u_\be + \nabla_\al (\la u_\be u^\mu) \nabla_\mu u^\al .
\end{split}
\label{5.2}
\end{align}

\subsection{Calculation of $\nabla_\al {}^{(3)}t^\al_\be$} Compute
\begin{align}
\begin{split}
\nabla_\al {}^{(3)}t^\al_\be & = 
\nabla_\al \big( \frac{1}{3} \pi^\al_\be \nabla_\mu u^\mu \big ) =
\frac{1}{3} \chi \pi^\al_\be \nabla_\al \nabla_\mu u^\mu 
+ \frac{1}{3} \nabla_\al (\chi \pi^\al_\be) \nabla_\mu u^\mu ,
\end{split}
\nonumber
\end{align}
so that
\begin{align}
\begin{split}
\nabla_\al {}^{(3)}t^\al_\be & =  
\chi \frac{1}{3} g^\mu_\be \nabla_\mu \nabla_\al u^\al 
+ \frac{1}{3} \chi u_\be u^\mu \nabla_\mu \nabla_\al u^\al 
+ {}^{(3)} B_\be, 
\end{split}
\label{5.3}
\end{align}
where
\begin{align}
\begin{split}
{}^{(3)} B_\be = \frac{1}{3} \nabla_\al (\chi \pi^\al_\be) \nabla_\mu u^\mu .
\end{split}
\label{5.4}
\end{align}

\subsection{Calculation of $\nabla_\al {}^{(4)}t^\al_\be$} Compute
\begin{align}
\begin{split}
\nabla_\al {}^{(4)}t^\al_\be & = \nabla_\al \big( \chi u^\al u_\be \nabla_\mu u^\mu \big )
\\ &
= \chi u_\be u^\mu \nabla_\mu \nabla_\al u^\al 
+ {}^{(4)}B_\be, 
\end{split}
\label{5.5}
\end{align}
where
\begin{align}
\begin{split}
{}^{(4)}B_\be & = \nabla_\al (\chi u^\al u_\be ) \nabla_\mu u^\mu.
\end{split}
\label{5.6}
\end{align}

\subsection{Calculation of $\nabla_\al {}^{(5)}t^\al_\be$} Compute
\begin{align}
\begin{split}
\nabla_\al {}^{(5)}t^\al_\be & = 
\nabla_\al \Big[ \frac{\la}{4\ep} (u^\al \pi^\mu_\be \nabla_\mu \ep 
+ u_\be \pi^{\al\mu} \nabla_\mu \ep \Big]
\\ &
= \frac{\la}{4\ep} u^\al \pi_\be^\mu \nabla_\al \nabla_\mu \ep 
+\frac{\la}{4\ep} u_\be \pi^{\al \mu} \nabla_\al \nabla_\mu \ep 
+ \nabla_\al \big [ \frac{\la}{4\ep} (u^\al \pi^\mu_\be + u_\be \pi^{\al \mu} ) \big ]
\nabla_\mu \ep 
\\ &
= \frac{\la}{4 \ep} u^\al g^\mu_\be \nabla_\al \nabla_\mu \ep 
+ \frac{\la}{4\ep} u_\be u^\al u^\mu \nabla_\al \nabla_\mu \ep 
+ \frac{\la }{4\ep } u_\be g^{\al\mu} \nabla_\al \nabla_\mu \ep 
\\
&
+ \frac{\la}{4\ep} u_\be u^\al u^\mu \nabla_\al \nabla_\mu \ep 
+  \nabla_\al \big [ \frac{\la}{4\ep} (u^\al \pi^\mu_\be + u_\be \pi^{\al \mu} ) \big ]
\nabla_\mu \ep .
\end{split}
\nonumber
\end{align}
We rearrange the terms, swapping the first and third terms, 
so that 
\begin{align}
\begin{split}
\nabla_\al {}^{(5)}t^\al_\be & = 
\frac{\la }{4\ep } u_\be g^{\al\mu} \nabla_\al \nabla_\mu \ep 
+ \frac{\la}{4\ep} u_\be u^\al u^\mu \nabla_\al \nabla_\mu \ep 
\\ &
+\frac{\la}{4 \ep} u^\al g^\mu_\be \nabla_\al \nabla_\mu \ep 
+ \frac{\la}{4\ep} u_\be u^\al u^\mu \nabla_\al \nabla_\mu \ep 
+ {}^{(5)}B_\be ,
\end{split}
\label{6.1}
\end{align}
where 
\begin{align}
\begin{split}
{}^{(5)}B_\be  & = 
 \nabla_\al \big [ \frac{\la}{4\ep} (u^\al \pi^\mu_\be + u_\be \pi^{\al \mu} ) \big ]\nabla_\mu \ep .
\end{split}
\label{6.2}
\end{align}

\subsection{Calculation of $\nabla_\al {}^{(6)}t^\al_\be$} Compute
\begin{align}
\begin{split}
\nabla_\al {}^{(6)}t^\al_\be & = 
\nabla_\al \Big[ \frac{3\chi}{4\ep} u^\al u_\be u^\mu \nabla_\mu \ep \Big]
= \frac{3\chi}{4\ep} u_\be u^\al u^\mu \nabla_\al \nabla_\mu \ep 
+ \nabla_\al \Big[ \frac{3\chi}{4\ep} u^\al u_\be u^\mu \Big] \nabla_\mu \ep 
\\ &
= \frac{3\chi}{4\ep} u_\be u^\al u^\mu \nabla_\al \nabla_\mu \ep 
+ {}^{(6)}B_\be ,
\end{split}
\label{6.3}
\end{align}
where 
\begin{align}
\begin{split}
{}^{(6)}B_\be & =
\nabla_\al \Big[ \frac{3\chi}{4\ep} u^\al u_\be u^\mu \Big] \nabla_\mu \ep .
\end{split}
\label{6.4}
\end{align}

\subsection{Calculation of $\nabla_\al {}^{(7)}t^\al_\be$} Compute
\begin{align}
\begin{split}
\nabla_\al {}^{(7)}t^\al_\be & =
\nabla_\al \Big[ \frac{\chi}{4\ep} \pi^\al_\be u^\mu \nabla_\mu \ep \Big ] 
= \frac{\chi}{4\ep} (g^\al_\be +u^\al u_\be) u^\mu \nabla_\al \nabla_\mu \ep 
+ \nabla_\al \Big[ \frac{\chi}{4\ep} \pi^\al_\be u^\mu \Big] \nabla_\mu \ep 
\\ &
= \frac{\chi}{4\ep}g^\al_\be u^\mu \nabla_\al \nabla_\mu \ep 
+ \frac{\chi}{4\ep} u^\al u_\be u^\mu \nabla_\al \nabla_\mu \ep 
+ {}^{(7)}B_\be, 
\end{split}
\label{7.1}
\end{align}
where
\begin{align}
{}^{(7)}B_\be & = 
+ \nabla_\al \Big[ \frac{\chi}{4\ep} \pi^\al_\be u^\mu \Big] \nabla_\mu \ep .
\label{7.2}
\end{align}

\subsection{Calculation of $\nabla_\al T^\al_\be$} Using
\eqref{conformal_tensor}, 
\eqref{ideal_tensor}, \eqref{4.1}, \eqref{5.1}, \eqref{5.3}, \eqref{5.5}, 
\eqref{6.1}, \eqref{6.3}, and \eqref{7.1}, we find
\begin{align}
\begin{split}
\nabla_\al T^\al_\be &=
-\eta( g^{\al\mu} + u^\al u^\be ) \nabla_\al \nabla_\mu u_\be
-\frac{1}{3} \eta g^\nu_\be \nabla_\nu \nabla_\al u^\al 
-\frac{1}{3} \eta u_\be u^\nu \nabla_\nu \nabla_\al u^\al 
\\&
+ \la u^\al u^\mu \nabla_\al \nabla_\mu u_\be 
+ \la u_\be u^\mu \nabla_\mu \nabla_\al u^\al 
\\&
+ \chi \frac{1}{3} g^\mu_\be \nabla_\mu \nabla_\al u^\al 
+ \frac{1}{3} \chi u_\be u^\mu \nabla_\mu \nabla_\al u^\al 
\\&
+ \chi u_\be u^\mu \nabla_\mu \nabla_\al u^\al 
+ \frac{\la }{4\ep } u_\be g^{\al\mu} \nabla_\al \nabla_\mu \ep 
+ \frac{\la}{4\ep} u_\be u^\al u^\mu \nabla_\al \nabla_\mu \ep 
\\ &
+\frac{\la}{4 \ep} u^\al g^\mu_\be \nabla_\al \nabla_\mu \ep 
+ \frac{\la}{4\ep} u_\be u^\al u^\mu \nabla_\al \nabla_\mu \ep 
+ \frac{3\chi}{4\ep} u_\be u^\al u^\mu \nabla_\al \nabla_\mu \ep 
\\&
+\frac{\chi}{4\ep} g^\al_\be u^\mu \nabla_\al \nabla_\mu \ep 
+ \frac{\chi}{4\ep} u^\al u_\be u^\mu \nabla_\al \nabla_\mu \ep 
 + B_\be,
\end{split}
\label{8.1}
\end{align}
where the first three terms on the RHS of \eqref{8.1} 
come from \eqref{4.1}, the fourth and fifth from \eqref{5.1},
the sixth and seventh from \eqref{5.3}, the eighth from \eqref{5.5}, the ninth to twelfth 
from \eqref{6.1}, the thirteenth from \eqref{6.3}, the fourteenth and fifteenth from \eqref{7.1},
and $B_\be$ is given by
\begin{align}
\begin{split}
B_\be  = {}^{(1)} B_\be +  {}^{(2)} B_\be +  {}^{(3)} B_\be +  {}^{(4)} B_\be +  {}^{(5)} B_\be + 
 {}^{(6)} B_\be +  {}^{(7)} B_\be + \nabla_\al {}^{(0)}t^\al_\be,
\end{split}
\label{8.2}
\end{align}
with $ {}^{(1)} B_\be,\dots,  {}^{(7)} B_\be$ given by 
\eqref{4.2}, \eqref{5.2}, \eqref{5.4}, \eqref{5.6}, \eqref{6.2},
\eqref{6.4}, and \eqref{7.2}, respectively, and ${}^{(0)}t^\al_\be$ is given by
\eqref{ideal_tensor}. We now group the terms on the RHS of \eqref{8.1} as follows:
the first and the fourth terms, 
the fifth and the eighth terms,
the second and the sixth terms,
the third and the seventh terms,
the ninth, tenth, and thirteenth terms,
the eleventh and fourteenth terms, 
and the twelfth and fifteenth terms. We obtain:
\begin{align}
\begin{split}
\nabla_\al T^\al_\be &=
(-\eta g^{\al\mu} + (\la -\eta) u^\al u^\mu ) \nabla_\al \nabla_\mu u_\be 
+ (\la +\chi) u_\be u^\mu \nabla_\mu \nabla_\al u^\al 
\\ &
 + \frac{1}{3} (-\eta + \chi) g^\mu_\be \nabla_\mu \nabla_\al u^\al 
+ \frac{1}{3} (-\eta + \chi) u_\be u^\mu \nabla_\mu \nabla_\al u^\al 
\\ &
+ \frac{1}{4\ep} u_\be (\la g^{\al\mu} + (\la + 3\chi) u^\al u^\mu ) \nabla_\al \nabla_\mu \ep 
+ \frac{1}{4\ep} (\la + \chi) g^\mu_\be u^\al \nabla_\al \nabla_\mu \ep 
\\ &
+ \frac{1}{4\ep} (\la + \chi) u_\be u^\al u^\mu \nabla_\al \nabla_\mu \ep 
+ B_\be , 
\end{split}
\label{8.3}
\end{align}
where the first term on the RHS of \eqref{8.3} comes from 
the first and the fourth terms on the RHS of \eqref{8.1}, 
the second term on the RHS of \eqref{8.3} comes from 
the fifth and the eighth terms on the RHS of \eqref{8.1},
the third term on the RHS of \eqref{8.3} comes from 
second and the sixth terms on the RHS of \eqref{8.1},
the fourth term on the RHS of \eqref{8.3} comes from 
the third and the seventh terms on the RHS of \eqref{8.1},
the fifth term on the RHS of \eqref{8.3} comes from 
the ninth, tenth, and thirteenth terms on the RHS of \eqref{8.1},
the sixth term on the RHS of \eqref{8.3} comes from 
the eleventh and fourteenth terms on the RHS of \eqref{8.1}, 
the seventh term on the RHS of \eqref{8.3} comes from 
and the twelfth and fifteenth terms on the RHS of \eqref{8.1}, and we used
that $\nabla_\al \nabla_\mu \ep = \nabla_\mu \nabla_\al \ep$.

Expanding the covariant derivatives and using Notation \ref{notation_B} gives
(\ref{div_T_gauge}).

\subsection{Derivation of (\ref{EE_gauge})} Let us first write (\ref{EE}) in trace reversed
form. Tracing (\ref{EE}) gives
\begin{gather}
R = 4\Lambda - T,
\nonumber
\end{gather}
where $T = g^{\al\be} T_{\al\be}$. (For (\ref{conformal_tensor}) we in fact have 
$T=0$, as it must be for a conformal tensor. But at this point we are 
writing Einstein's equations for a general tensor.) Plugging this for $R$ in (\ref{EE}) gives
\begin{gather}
R_{\al\be} = T_{\al\be} -\frac{1}{2} T g_{\al\be} + \Lambda g_{\al\be}.
\nonumber
\end{gather}
We now proceed to compute $R_{\al\be}$ in local coordinates. In coordinates, we have
\begin{align}
\begin{split}
R_{\al\be} & = \partial_\la \Ga^\la_{\al\be} - \partial_\al \Ga_{\be\la}^\la 
+ \Ga_{\al\be}^\la \Ga_{\la \mu}^\mu - \Ga_{\al\mu}^\la \Ga_{\be\la}^\mu.
\end{split}
\nonumber
\end{align}
Using the definition of the Christoffel symbols $\Ga_{\al\be}^\la$ gives
\begin{align}
\begin{split}
R_{\al\be} & = -\frac{1}{2} g^{\mu\nu}\partial^2_{\mu\nu} g_{\al\be} + \frac{1}{2} (
g_{\al\la }\partial_\be \Ga^\la + g_{\be\la} \partial_\al \Ga^\la )
\\ &
-\frac{1}{2}(\partial_\be g^{\la \mu} \partial_\la g_{\al \mu} 
+ \partial_\al g^{\la \mu} \partial_\la g_{\be\mu} )
-\Ga_{\al\la}^\mu\Ga_{\be\mu}^\la ,
\end{split}
\nonumber
\end{align}
where $\Ga^\la$ is given by
\begin{gather}
\Ga^\la = g^{\mu \nu} \Ga_{\mu\nu}^\la.
\nonumber
\end{gather}
Using that in wave coordinates $\Ga^\la = 0$ and recalling (\ref{conformal_tensor}),
the above gives (\ref{EE_gauge}).

\section{The characteristic determinant}
\label{section_char_det}

In this section we derive (\ref{char_det}).
Because of the structure of the system in (\ref{main_system}) 
it suffices to compute the characteristic determinant of $m(U,\partial)$
in (\ref{matrix_M}). Using Mathematica and (\ref{proportional_eta})
 we find (we are not assuming $a_1 =4$ at this point)
\begin{align}
\begin{split}
\det m(\widehat{U}, \xi) & = \widetilde{p}_1(\widehat{U}, \xi)
\widetilde{p}_2(\widehat{U}, \xi) \widetilde{p}_3(\widehat{U}, \xi),
\end{split}
\nonumber
\end{align}
where
\begin{align}
\begin{split}
\widetilde{p}_1(\widehat{U}, \xi) 
= \frac{1}{12 \widehat{\ep}} \eta^4 (\widehat{u}^\mu \xi_\mu)^2,
\end{split}
\nonumber
\end{align}
\begin{align}
\begin{split}
\widetilde{p}_2(\widehat{U}, \xi) 
= & \,
\left[ (a_2-1)  (\widehat{u}^0)^2 \xi_0^2 + (a_2-1) (\widehat{u}^1)^2 \xi_1^2 
-(\widehat{u}^2)^2 \xi_2^2
+a_2 (\widehat{u}^2)^2 \xi_2^2 -2 \widehat{u}^2 \widehat{u}^3 \xi_2 \xi_3
 \right.
 \\
  & 
  + 2a_2 \widehat{u}^2 \widehat{u}^3 \xi_2 \xi_3
 - (\widehat{u}^3)^2 \xi_3^2 + a_2 (\widehat{u}^3)^2 \xi_3^2
 + \xi_0( 2(-1 +a_2)\xi_1 \widehat{u}^0 \widehat{u}^1 
\\
& 
 + 2(a_2 -1) \xi_2 \widehat{u}^0 \widehat{u}^2
  -2 \xi_3 \widehat{u}^0 \widehat{u}^3 
+ 2  a_2 \widehat{u}^0 \widehat{u}^3 \xi_3- \xi^0 )
\\
 & 
 \left.
 + \xi_1( 2(-1 +a_2)\widehat{u}^1 \widehat{u}^2 \xi_2  + 2(a_2-1) \widehat{u}^1 \widehat{u}^3\xi_3 
 - \xi^1 )
 -\xi_2 \xi^2 - \xi_3 \xi^3
\right]^2
\end{split}
\nonumber
\end{align}
and
\begin{align}
\begin{split}
& \widetilde{p}_3(\widehat{U}, \xi) =
\\ & 
-6 (-2 a_1 \widehat{u}_0 \widehat{u}^0 - 
    a_2 \widehat{u}_0 \widehat{u}^0 + 
    2 a_1 a_2 \widehat{u}_0 \widehat{u}^0 + 
    a_2^2 \widehat{u}_0 \widehat{u}^0
     - 
    2 a_1 \widehat{u}_1 \widehat{u}^1 - 
    a_2 \widehat{u}_1 \widehat{u}^1 
    \\
    &
    + 
    2 a_1 a_2 \widehat{u}_1 \widehat{u}^1 + 
    a_2^2 \widehat{u}_1 \widehat{u}^1 - 
    2 a_1 \widehat{u}_2 \widehat{u}^2 - 
    a_2 \widehat{u}_2 \widehat{u}^2 + 
    2 a_1 a_2 \widehat{u}_2 \widehat{u}^2 + 
    a_2^2 \widehat{u}_2 \widehat{u}^2  
	\\
	&    
    -2 a_1 \widehat{u}_3 \widehat{u}^3 - 
    a_2 \widehat{u}_3 \widehat{u}^3 + 
    2 a_1 a_2 \widehat{u}_3 \widehat{u}^3 + 
    a_2^2 \widehat{u}_3 \widehat{u}^3) (\xi_0
      \widehat{u}^0 + \xi_1 \widehat{u}^1 
    \\
    &
    + \xi_2 \widehat{u}^2 + 
    \xi_3 \widehat{u}^3)^4
    \\
    & 
    - 2 (-a_2 \widehat{u}_0 + 4 a_1 a_2 \widehat{u}_0 + 
    3 a_2^2 \widehat{u}_0) (\xi_0 \widehat{u}^0 + 
    \xi_1 \widehat{u}^1 + 
    \xi_2 \widehat{u}^2 + 
    \xi_3 \widehat{u}^3)^3 \xi^0
    \\
    &
     - 2 (-a_2 \widehat{u}_1 + 4 a_1 a_2 \widehat{u}_1 + 
    3 a_2^2 \widehat{u}_1) (\xi_0 \widehat{u}^0 + 
    \xi_1 \widehat{u}^1 + 
    \xi_2 \widehat{u}^2 + 
    \xi_3 \widehat{u}^3)^3 \xi^1
    \\
    & 
    -2 (-a_2 \widehat{u}_2 + 4 a_1 a_2 \widehat{u}_2 + 
    3 a_2^2 \widehat{u}_2) (\xi_0 \widehat{u}^0 + 
    \xi_1 \widehat{u}^1 + 
    \xi_2 \widehat{u}^2 + 
    \xi_3 \widehat{u}^3)^3 \xi^2 
    \\
    & 
    -2 (-a_2 \widehat{u}_3 + 4 a_1 a_2 \widehat{u}_3 + 
    3 a_2^2 \widehat{u}_3) (\xi_0 \widehat{u}^0 + 
    \xi_1 \widehat{u}^1 + 
    \xi_2 \widehat{u}^2 + 
    \xi_3 \widehat{u}^3)^3 \xi^3 
    \\
    &
     + 5 (3 a_1 \widehat{u}_0 \widehat{u}^0 + 
    2 a_2 \widehat{u}_0 \widehat{u}^0 + 
    a_1 a_2 \widehat{u}_0 \widehat{u}^0 + 
    3 a_1 \widehat{u}_1 \widehat{u}^1 + 
    2 a_2 \widehat{u}_1 \widehat{u}^1 
    \\
    &
    + a_1 a_2 \widehat{u}_1 \widehat{u}^1 + 
    3 a_1 \widehat{u}_2 \widehat{u}^2 + 
    2 a_2 \widehat{u}_2 \widehat{u}^2 + 
    a_1 a_2 \widehat{u}_2 \widehat{u}^2 
    \\
    & 
    + 3 a_1 \widehat{u}_3 \widehat{u}^3 + 
    2 a_2 \widehat{u}_3 \widehat{u}^3 + 
    a_1 a_2 \widehat{u}_3 \widehat{u}^3) (\xi_0
      \widehat{u}^0 + \xi_1 \widehat{u}^1 
    \\
    &
    +\xi_2 \widehat{u}^2 + 
    \xi_3 \widehat{u}^3)^2 (\xi_0
      \xi^0 + 
    \xi_1 \xi^1
     + \xi_2 \xi^2 + 
    \xi_3 \xi^3) 
    \\
    &
    + (3 a_1 \widehat{u}_0 + 2 a_2 \widehat{u}_0 + 
    a_1 a_2 \widehat{u}_0) (\xi_0 \widehat{u}^0 + 
    \xi_1 \widehat{u}^1 + 
    \xi_2 \widehat{u}^2 + 
    \xi_3 \widehat{u}^3) \xi^0 (\xi_0 \xi^0 
    \\
    &
    + \xi_1 \xi^1
     + \xi_2 \xi^2 + 
    \xi_3 \xi^3) + (3 a_1 \widehat{u}_1 + 2 a_2 \widehat{u}_1 + 
    a_1 a_2 \widehat{u}_1) (\xi_0 \widehat{u}^0 
    \\
    &
    + 
    \xi_1 \widehat{u}^1 + 
    \xi_2 \widehat{u}^2
    + \xi_3 \widehat{u}^3) \xi^1 (\xi_0 \xi^0 + 
    \xi_1 \xi^1 + 
    \xi_2 \xi^2 + 
    \xi_3 \xi^3)
    \\
    &
     + (3 a_1 \widehat{u}_2 + 2 a_2 \widehat{u}_2 + 
    a_1 a_2 \widehat{u}_2) (\xi_0 \widehat{u}^0 + 
    \xi_1 \widehat{u}^1 
    \\
    &
    + \xi_2 \widehat{u}^2 + 
    \xi_3 \widehat{u}^3) \xi^2 (\xi_0 \xi^0 + 
    \xi_1 \xi^1 + 
    \xi_2 \xi^2 + 
    \xi_3 \xi^3) 
    \\
    &
    + (3 a_1 \widehat{u}_3 + 2 a_2 \widehat{u}_3 + 
    a_1 a_2 \widehat{u}_3) (\xi_0 \widehat{u}^0 
    \\
    &
    + \xi_1 \widehat{u}^1 + 
    \xi_2 \widehat{u}^2 + 
    \xi_3 \widehat{u}^3) \xi^3 (\xi_0 \xi^0 + 
    \xi_1 \xi^1 + 
    \xi_2 \xi^2 + 
    \xi_3 \xi^3) 
    \\
    &
    + (4 a_2 \widehat{u}_0
      \widehat{u}^0 - a_1 a_2 \widehat{u}_0 \widehat{u}^0 + 
    4 a_2 \widehat{u}_1 \widehat{u}^1 - 
    a_1 a_2 \widehat{u}_1 \widehat{u}^1 + 
    4 a_2 \widehat{u}_2 \widehat{u}^2 
    \\
    &
    - a_1 a_2 \widehat{u}_2 \widehat{u}^2 + 
    4 a_2 \widehat{u}_3 \widehat{u}^3 - 
    a_1 a_2 \widehat{u}_3 \widehat{u}^3) (\xi_0
      \xi^0 + 
    \xi_1 \xi^1 + 
    \xi_2 \xi^2 + 
    \xi_3 \xi^3)^2.
\end{split}
\label{p_3_tilde}
\end{align}
It is not difficult to see, after some manipulations, that $\widetilde{p}_2(\widehat{U},\xi)$
is precisely $p_2(\widehat{U},\xi)$, i.e., (\ref{p2}). Let us now analyze 
$\widetilde{p}_3(\widehat{U},\xi)$. The first term in $\widetilde{p}_3(\widehat{U},\xi)$, that
spans lines 2 to 5 in (\ref{p_3_tilde}), is proportional to $(\widehat{u}^\mu \xi_\mu)^4$.
The terms from lines 6 to 9 combined are also proportional to $(\widehat{u}^\mu \xi_\mu)^4$. Indeed,
the term on the sixth line can be written as
\begin{align}
\begin{split}
- 2 (-a_2 \widehat{u}_0 + 4 a_1 a_2 \widehat{u}_0 + 
    3 a_2^2 \widehat{u}_0) (\xi_0 \widehat{u}^0 + 
    \xi_1 \widehat{u}^1 + 
    \xi_2 \widehat{u}^2 + 
    \xi_3 \widehat{u}^3)^3 \xi^0
    \\
    =
    - 2 (-a_2  + 4 a_1 a_2  + 
    3 a_2^2 ) (\xi_0 \widehat{u}^0 + 
    \xi_1 \widehat{u}^1 + 
    \xi_2 \widehat{u}^2 + 
    \xi_3 \widehat{u}^3)^3 \widehat{u}_0\xi^0,
    \\
\end{split}
\nonumber
\end{align}
and similarly we can group $\widehat{u}_i$ with $\xi^i$ in the terms on the seventh to ninth line.
Factoring then the common factor in lines 6 to 9 gives a term cubic in 
$\widehat{u}^\mu \xi_\mu$ times the term
\begin{gather}
\widehat{u}_0\xi^0 + \widehat{u}_1\xi^1 + \widehat{u}_2\xi^2 + \widehat{u}_3\xi^3.
\nonumber
\end{gather}
But this last term equals $\widehat{u}^\mu \xi_\mu$, which can then be grouped
with the cubic term in $\widehat{u}^\mu \xi_\mu$ producing a term proportional to 
$(\widehat{u}^\mu \xi_\mu)^4$, as claimed. 

The next term in $\widetilde{p}_3(\widehat{U},\xi)$, spanning lines
10 to 13 in (\ref{p_3_tilde}) is proportional to $(\widehat{u}^\mu \xi_\mu)^2$.

We claim that the terms spanning lines 14 to 20, when combined, produce a term
proportional to $(\widehat{u}^\mu \xi_\mu)^2$. To see this, note that as written
the terms in lines 14 to 20 all have a factor $\widehat{u}_0\xi^0 + \widehat{u}_1\xi^1 + \widehat{u}_2\xi^2 + \widehat{u}_3\xi^3$, which equals $\widehat{u}^\mu \xi_\mu$. The term
that begins on line 14 of (\ref{p_3_tilde}) can be written as
\begin{align}
\begin{split}
 (3 a_1 \widehat{u}_0 + 2 a_2 \widehat{u}_0 + 
    a_1 a_2 \widehat{u}_0) (\xi_0 \widehat{u}^0 + 
    \xi_1 \widehat{u}^1 + 
    \xi_2 \widehat{u}^2 + 
    \xi_3 \widehat{u}^3) \xi^0 (\xi_0 \xi^0 
    + \xi_1 \xi^1
     + \xi_2 \xi^2 + 
    \xi_3 \xi^3)
    \\
    =
  (3 a_1  + 2 a_2  + 
    a_1 a_2 ) (\xi_0 \widehat{u}^0 + 
    \xi_1 \widehat{u}^1 + 
    \xi_2 \widehat{u}^2 + 
    \xi_3 \widehat{u}^3)  (\xi_0 \xi^0 
    + \xi_1 \xi^1
     + \xi_2 \xi^2 + 
    \xi_3 \xi^3)   \widehat{u}_0\xi^0,
\end{split}
\nonumber
\end{align}
and similarly we can combine $\widehat{u}_i$ with $\xi^i$ in the other terms
in lines 15 to 20. Factoring then the common factor to all terms in lines
14 to 20 produces a term linear in $\widehat{u}^\mu \xi_\mu$ times
$\widehat{u}_0\xi^0 + \widehat{u}_1\xi^1 + \widehat{u}_2\xi^2 + \widehat{u}_3\xi^3
\equiv u^\mu \xi_\mu$, hence a term quadratic in $\widehat{u}^\mu \xi_\mu$, as claimed.

Therefore, we see that all terms in $\widetilde{p}_3(\widehat{U},\xi)$ contain
a factor of $(\widehat{u}^\mu \xi_\mu)^2$, except for the last term which spans lines
21 and 22. This last term, however, vanishes identically if $a_1 = 4$. In this case we can
factor $(\widehat{u}^\mu \xi_\mu)^2$ from $\widetilde{p}_3(\widehat{U},\xi)$. We
combine the factored $(\widehat{u}^\mu \xi_\mu)^2$ with
$\widetilde{p}_1(\widehat{U},\xi)$, producing $p_1(\widehat{U},\xi)$, i.e., 
(\ref{p1}), and the remainder from $\widetilde{p}_3(\widehat{U},\xi)$ produces
$p_3(\widehat{U},\xi)$, i.e., (\ref{p3}).

\begin{remark}
Without setting $a_1=4$, the above factorization procedure can be used
to show that $\widetilde{p}_3(\widehat{U},\xi)$
factors as
\begin{gather}
A (\widehat{u}^\mu \xi_\mu)^4 + B (\widehat{u}^\mu \xi_\mu)^2
\xi^\la \xi_\la + C (\xi^\la \xi_\la)^2,
\nonumber
\end{gather}
where $A$, $B$, and $C$ depend on $a_1$ and $a_2$. We would like to factor this quartic polynomial
as a product of (real) degree two polynomials, since then we can analyze its roots explicitly.
The above choice $a_1=4$ does exactly this. But other choices of $a_1$ and $a_2$ also lead
to the desired factorization, as showed in \cite{BemficaDisconziNoronha}.
\label{remark_factorization}
\end{remark}

\bibliographystyle{plain}
\bibliography{References.bib}

\end{document}